\documentclass[a4paper,12pt,reqno]{amsart}
\usepackage{amssymb}
\usepackage{amsmath}
\usepackage{mathtools}
\usepackage{enumitem}
\usepackage{mathrsfs}
\usepackage[all]{xy}
\usepackage{mathtools}
\usepackage{hyperref}
\usepackage{url}
\usepackage{bbm}
\usepackage{longtable}

\usepackage[OT2,T1]{fontenc}
\DeclareSymbolFont{cyrletters}{OT2}{wncyr}{m}{n}
\usepackage[british]{babel}

\usepackage[margin=1.3in]{geometry}
\numberwithin{equation}{section} \numberwithin{figure}{section}

\DeclareMathOperator{\Gal}{Gal}

\DeclareMathOperator{\Hom}{Hom} \DeclareMathOperator{\re}{Re}
\DeclareMathOperator{\im}{Im}

\DeclareMathOperator{\Res}{R}

\DeclareMathOperator{\HH}{H}
\DeclareMathOperator{\spl}{spl}

\DeclareMathOperator{\Frob}{Frob}

\newcommand{\places}{\Omega_k}

\DeclareMathSymbol{\Sha}{\mathalpha}{cyrletters}{"58}

\newcommand{\OO}{\mathcal{O}}
\let \O \relax
\newcommand{\O}{\mathcal{O}}

\newcommand{\dual}[1]{{#1}^{\wedge}}

\newcommand\FF{\mathbb{F}}

\newcommand\ZZ{\mathbb{Z}}

\newcommand\QQ{\mathbb{Q}}
\newcommand\RR{\mathbb{R}}
\newcommand\CC{\mathbb{C}}

\newcommand\Z{\mathbb{Z}}

\newcommand\Q{\mathbb{Q}}

\newcommand\GG{\mathbb{G}}

\newcommand\Gm{\GG_\mathrm{m}}

\newcommand{\Adele}{\mathbf{A}}

\newcommand{\Idele}{{\mathbf{A}^{*}}}

\newcommand{\pair}[2]{\ensuremath{\langle #1, #2 \rangle}}

\newtheorem{lemma}{Lemma}

\newtheorem{theorem}[lemma]{Theorem}
\newtheorem{proposition}[lemma]{Proposition}

\theoremstyle{definition}
\newtheorem{example}[lemma]{Example}

\newtheorem{remark}[lemma]{Remark}

\newtheorem*{ack}{Acknowledgements}

\begin{document}

\title{The Hasse Norm principle for abelian extensions -- corrigendum}

\author{\sc Christopher Frei}
\address{Christopher Frei\\
Graz University of Technology\\
Institute of Analysis and Number Theory\\
Steyrergasse 30/II\\
8010 Graz\\
Austria.}
\email{frei@math.tugraz.at}
\urladdr{https://www.math.tugraz.at/~frei/}

\author{Daniel Loughran}
 \address{Daniel Loughran \\
	Department of Mathematical Sciences\\
	University of Bath\\
	Claverton Down\\
	Bath\\
	BA2 7AY\\
	UK.}
\urladdr{https://sites.google.com/site/danielloughran/}

\author{\sc Rachel Newton}
\address{Rachel Newton\\
Department of Mathematics\\
King's College London\\ 
Strand\\
London WC2R 2LS\\
UK.}
   \email{rachel.newton@kcl.ac.uk}
\urladdr{https://sites.google.com/view/rachelnewton}

\begin{abstract}
  The proofs of Theorem 1.1 and Theorem 1.5(2) in the authors' paper \emph{The Hasse norm principle for abelian extensions} are incorrect. We point out the mistakes and provide correct proofs, using techniques of the original paper.
\end{abstract}

\maketitle

\tableofcontents

\section{Introduction}

\subsection{The mistakes}
This is a corrigendum to the paper \cite{HNP}. The paper contains the following independent mistakes:
\begin{enumerate}
	\item\label{5.1} The proof of Theorem 5.1 has a gap in the dominated convergence argument, since the last part of Lemma 5.3 is false.
	\item In Lemma 6.7 the implication `$\implies$' is false. (A corrected statement is Lemma \ref{lem:WA} in this document.)
	\item\label{6.9} The final statement in Lemma 6.9 is false. Therefore Lemma 6.12 is false.
	\item\label{typo} Typo: sign error in equation (4.15). The Euler factor should start $1+(Q^\beta -1)q_v^{-1}$ in the case $q\equiv 1 \pmod Q$.
\end{enumerate}

The typo~\eqref{typo} is inconsequential but mistakes~\eqref{5.1}--\eqref{6.9} have the following consequences for the paper:
\begin{itemize}
	\item The proof of Theorem 1.1 is incomplete (it uses Theorem 5.1 and Lemma 6.12).
	\item The proof of Theorem 1.5(2) is incomplete (it uses Lemma 6.7).
	\item The proof of Theorem 5.2(2) is incomplete (it uses Theorem 5.1).
\end{itemize}

\begin{example}
	Here is an explicit counter-example to Lemma 6.7 and Lemma 6.9. Consider the number field $K/\QQ$
	given in \cite[\href{http://www.lmfdb.org/NumberField/8.0.10070523904.2}{Number Field 8.0.10070523904.2}]{LMFDB}. This is a $\ZZ/4\ZZ \times \ZZ/2\ZZ$ extension of $\QQ$ whose only non-cyclic decomposition group occurs
	at $7$ and is $(\ZZ/2\ZZ)^2$. However, one checks that the map
	$$\ZZ/2\ZZ = \HH^3(\ZZ/4\ZZ \times \ZZ/2\ZZ,\ZZ) \to \HH^3((\ZZ/2\ZZ)^2,\ZZ) = \ZZ/2\ZZ$$
	is the zero map, which implies that weak approximation holds. We are grateful to Andr\'{e} Macedo for providing this illustrative example.
\end{example}

We remark that none of these issues affect the subsequent paper \cite{HNP2} which involved only the Hasse norm principle (no weak approximation) 
and counting by conductor, which is substantially easier than counting by discriminant with the harmonic analysis approach.

\subsection{The fix} We focus on providing complete correct proofs for the main
results (Theorems 1.1 and 1.5(2)) from the introduction of~\cite{HNP}. Their statements are as follows. Let $k$ be a number field. 

\begin{theorem}[{\cite[Theorem 1.1]{HNP}}]\label{thm1.1}
 Let $n,r\in\ZZ$ with $n > 1$ and $r \geq 0$. Let $Q$ be the smallest prime dividing $n$ and let $G = \ZZ/n\ZZ\oplus (\ZZ/Q\ZZ)^r$. Then 0\% of $G$-extensions of $k$ fail the Hasse norm principle, when ordered by discriminant.
\end{theorem}

\begin{theorem}[{\cite[Theorem 1.5(2)]{HNP}}] \label{thm1.5(2)}
Let $G$ be a non-trivial finite abelian group, let $Q$ be the smallest prime dividing $\lvert G\rvert$ and suppose that the $Q$-Sylow subgroup of $G$ is not cyclic. Then as $\varphi$ varies over all $G$-extensions of $k$, ordered by discriminant, 0\% of the tori $\Res_{K_\varphi/k}^1 \Gm$ satisfy weak approximation.
\end{theorem}

Proving these results turns out to be
actually quite delicate; the dominated convergence argument used in
the proof of \cite[Theorem 5.1]{HNP} is fatally flawed and cannot be rescued. Instead, we prove Theorems~\ref{thm1.1} and \ref{thm1.5(2)} by exhibiting explicit cancellation between the poles of different Dirichlet series which arise through M\"{o}bius inversion. Koymans and Rome have recently found alternative proofs of \cite[Theorem~1.1]{HNP}, see \cite{KR2} (for $k$ an arbitrary number field), and \cite[Theorem 1.5]{HNP}, see \cite{KR1} (only for $k = \Q$).

We now state the result which will imply both theorems. Let $k$ be a
number field and let $\Omega_k$ denote the set of all places of
$k$. We fix embeddings $\bar{k} \subset \bar{k}_v$ for all
$v\in\Omega_k$. For any non-archimedean place $v$, let
$I_v\subset\Gal(\bar{k}_v/k_v)$ be the inertia subgroup. Then the
coset $\Frob_v I_v$ is independent of the choice of Frobenius element
$\Frob_v\in \Gal(\bar{k}_v/k_v)$.

We will write our groups multiplicatively. For $n\in\ZZ$, we denote the group of $n$-th roots of unity in $\CC$ by $\mu_n$. Let $G$ be a finite abelian group and $Q$ the smallest prime dividing the order of $G$. For cyclic $G$, Theorem~\ref{thm1.1} follows from the Hasse norm theorem. Therefore, we assume throughout this corrigendum that the $Q$-Sylow subgroup of $G$ is not cyclic; this means that we may write $G=M\times \mu_{Q^{a_1}}\times \dots \times \mu_{Q^{a_t}}$ where $Q\nmid |M|$ and $t\geq 2$. Let $e_1,\dots ,e_t$ generate $\mu_{Q^{a_1}}, \dots, \mu_{Q^{a_t}}$, respectively.

As in~\cite{HNP}, by a \emph{sub-$G$-extension} of a field $F$ with separable closure $\bar{F}$, we mean a continuous homomorphism $\varphi: \Gal(\bar{F}/F)\to G$, and we call a surjective sub-$G$-extension of $F$ a \emph{$G$-extension} of $F$. 
Hence, a sub-$G$-extension $\varphi$ of $k$ induces a sub-$G$-extension $\varphi_v$ of $k_v$ at every place $v$. 

\begin{theorem}\label{thm:countWA}
Let $S$ be a finite set of places of $k$ containing all the archimedean places.
Let $i, j \in \{1,\ldots,t\}$ be distinct. For $v \notin S$, define $\Lambda_v$ to be the set of all sub-$G$-extensions $\varphi_v$ of $k_v$ for which the following implication holds:
\[
e_i^{Q^{a_i - 1}} \in \varphi_v(I_v) \implies 
\varphi_v(\Frob_v I_v) \subseteq \langle M, e_1,e_2,\dots,e_{j-1},e_j^Q,e_{j+1},\dots,e_t\rangle.
\]
Then only $0\%$ of all $G$-extensions $\varphi$ of $k$ satisfy $\varphi_v\in \Lambda_v$ for all $v\in \Omega_k\setminus S$, when ordered by discriminant.
\end{theorem}

We use this to prove our results as follows.

\subsection{Proof of Theorem~\ref{thm1.1}}
This theorem concerns groups of the form $G = M\times \mu_Q^{t-1}\times \mu_{Q^a}$, where $t\geq 2$, $a\geq 1$, $Q$ is a
prime, and $M$ is a cyclic group such that all prime divisors of $|M|$
are larger than $Q$.
 We need to show that only $0\%$ of $G$-extensions of
bounded discriminant fail the Hasse norm principle. We use the notation of Theorem~\ref{thm:countWA}.

If a $G$-extension $\varphi$ of $k$ fails the Hasse norm principle, then by \cite[Lemma
4.2]{HNP2} there is a maximal proper subgroup $\Upsilon$ of $\wedge^2 G$ such that,
for all places $v$ of $k$, the image of the natural map
$\wedge^2(\im\varphi_v)\to \wedge^2G$ is contained in $\Upsilon$. Hence, it
is enough to show that, for every such $\Upsilon$, this occurs for $0\%$ of all
$G$-extensions of $k$.

Henceforth, for a subgroup $H$ of $G$, we abuse notation by writing $\wedge^2 H\subseteq \Upsilon$ to mean that the image of the natural map
$\wedge^2H\to \wedge^2G$ is contained in $\Upsilon$.

\begin{lemma}\label{lem:relaxed_condition_no_cancellation}
  Let $\Upsilon\subset \wedge^2 G$ be a maximal proper subgroup. Then
  there exist
   $e_1,\dots , e_t\in G$ such that $e_t$ has order $Q^a$, $e_1,\dots, e_{t-1}$ have order $Q$ and $e_1,\dots, e_t$ generate $G/M$, and 
  indices
  $i,j\in\{1,\ldots,t\}$ with $i \leq t-1$ and 
  the following property: 
  for every finite place $v$ of $k$ and $G$-extension $\varphi$ of $k$ 
  such that $\wedge^2( \im \varphi_v)\subseteq \Upsilon$,
  we have
\begin{equation}\label{eq:missingh_cond}
e_i \in \varphi_v(I_v) \implies \varphi_v(\Frob_v I_v)\subseteq \langle M, e_1, \dots, e_{j-1}, e_j^Q,e_{j+1} ,\dots , e_t \rangle.
\end{equation}
\end{lemma}

\begin{proof}
We split the proof into two cases. 
First assume that $\wedge^2(\mu_Q^{t-1})\not\subseteq \Upsilon$, so that there
are order $Q$ elements
$\epsilon_i,\epsilon_j\in \mu_Q^{t-1}$
such that $\epsilon_i\wedge \epsilon_j\notin\Upsilon$.
Define $L_i:=\{g\in G\mid \epsilon_i\wedge g\in \Upsilon\}$. Then, if $\epsilon_i \in \varphi_v(I_v)$, the assumption $\wedge^2(\im \varphi_v)\subseteq \Upsilon$ implies $\varphi_v(\Frob_v I_v)\subseteq L_i$. We will now find generators of $G$ for which $L_i$ becomes the set on the right-hand side of \eqref{eq:missingh_cond}.

Note that $M\subseteq L_i$ since $\epsilon_i\wedge M=1$. Furthermore, $L_i$ is
a proper subgroup of $G$ since $\epsilon_j\notin L_i$.  By maximality of
$\Upsilon$, we have $|(\wedge^2 G)/\Upsilon|=Q$ and hence
$|(\epsilon_i\wedge G)/(\epsilon_i\wedge L_i)|\mid Q$, as
$\epsilon_i\wedge L_i=\Upsilon\cap (\epsilon_i\wedge G)$. Since
$\epsilon_i\wedge \epsilon_j\notin \Upsilon$, we have
$\epsilon_i\wedge L_i\neq \epsilon_i\wedge G$ and hence
$|(\epsilon_i\wedge G)/(\epsilon_i\wedge L_i)|= Q$.  In the case $a\geq 2$,
suppose for contradiction that $L_i\subseteq M\times G[Q^{a-1}]$.
  Then
$\epsilon_i\wedge L_i\subseteq \epsilon_i\wedge G[Q^{a-1}]$ and since
\[|(\epsilon_i\wedge G)/(\epsilon_i\wedge L_i)|=Q=|(\epsilon_i\wedge
  G)/(\epsilon_i\wedge G[Q^{a-1}])|,\] it follows that
$\epsilon_i\wedge L_i= \epsilon_i\wedge G[Q^{a-1}]$. This contradicts the fact
that $\epsilon_i\wedge \epsilon_j\notin \Upsilon$. Therefore, $L_i$ contains an
element of order $Q^a$.
Therefore, both when $a\geq 2$ and when $a=1$, we deduce that $L_i\cong M \times\mu_Q^s\times \mu_{Q^a}$ for some $s\leq t-2$, because $L_i$ is a proper subgroup of $G$ containing $M$. Since $\epsilon_i\in L_i$, it follows that $\epsilon_i\wedge L_i\cong \FF_Q^{s}$. Now $\epsilon_i\wedge G \cong \FF_Q^{t}$ and $|(\epsilon_i\wedge G)/(\epsilon_i\wedge L_i)|=Q$ together yield $s=t-1$. Thus, 
\begin{equation}\label{eq:L in G}
G\cong \mu_Q\times L_i.
\end{equation}
Fix a choice of isomorphism as in~\eqref{eq:L in G} and use it to identify $G$ with $ \mu_Q\times L_i$. Now choosing generators $e_1$ for the `extra' copy of $\mu_Q$ on the right-hand side of \eqref{eq:L in G} and $\epsilon_i=e_2,\dots,e_{t}$ for the $Q$-primary part of $L_i$ (where $e_1,\dots,e_{t-1}$ each have order $Q$ and $e_t$ has order $Q^a$) gives a choice of generators $e_1,\dots,e_{t}$ for the $Q$-primary part of $G$ such that $L_i=\langle M, e_2,\dots, e_t\rangle$, completing the proof in this case (with $(i,j)=(2,1)$).

Now assume that $\wedge^2(\mu_Q^{t-1})\subseteq \Upsilon$. Let $e_1,\dots, e_{t-1}$ generate $\mu_Q^{t-1}$ and let $e_t$ generate $\mu_{Q^a}$. Since $\Upsilon\neq \wedge^2G$ and the $e_m\wedge e_n$ for $1\leq m<n\leq t$ generate $\wedge^2G$, there exists $i\leq t-1$ such that $e_i\wedge e_t\notin\Upsilon$. Suppose $e_i \in \varphi_v(I_v)$. The assumption $\wedge^2(\im \varphi_v)\subseteq \Upsilon$ implies that \[\varphi_v(\Frob_v I_v)\subseteq L_i:=\{g\in G\mid e_i\wedge g\in \Upsilon\}.\] Note that $L_i$ is a proper subgroup of $G$ because $e_t\notin L_i$. Moreover, $L_i$ contains the maximal proper subgroup $\langle M, e_1, \dots , e_{t-1}, e_t^Q\rangle$ since $\wedge^2\mu_Q^{t-1}\subseteq \Upsilon$ and $e_i\wedge e_t^Q=e_i\wedge M=1$. Thus, $L_i=\langle M, e_1, \dots , e_{t-1}, e_t^Q\rangle$, completing the proof.
\end{proof}

Theorem~\ref{thm1.1} now follows immediately from Theorem \ref{thm:countWA}, noting that $a_i=1$ in the notation of Theorem \ref{thm:countWA}.

\subsection{Proof of Theorem~\ref{thm1.5(2)}}
We require the following corrected version of \cite[Lemma 6.7]{HNP}.

\begin{lemma} \label{lem:WA}
	Let $G$ be a finite abelian group and $\varphi$ a $G$-extension of $k$ with associated
	field $K$.
	Then $\Res_{K/k}^1 \Gm$
	satisfies weak approximation if and only if the induced map
	$$ \wedge^2 (\im \varphi_v) \to \wedge^2 G$$
	is the trivial map for all places $v$ of $k$. 
      \end{lemma}
      
\begin{proof}
	Results of Tate and Voskresenski\v{\i} (\cite[Theorems~6.1, 6.2]{HNP}) imply that weak approximation
	holds if and only if 
	$$ \mathrm{H}^3(G,\ZZ) = \ker\left(\mathrm{H}^3(G,\ZZ) \to \prod_v \mathrm{H}^3(\im \varphi_v,\ZZ)\right).$$
	Dualising and applying~\cite[Lemma 6.4]{HNP}, this is equivalent to
	$$\im\left( \prod_v \wedge^2 (\im\varphi_v) \to \wedge^2 G\right) =
        \{1\},$$
	which is equivalent to the statement of the lemma.
\end{proof}

We use this to prove Theorem~\ref{thm1.5(2)} as follows. This theorem concerns finite
abelian groups $G$ for which the $Q$-Sylow subgroup is not cyclic. Such groups can be written in the form $G=M\times \mu_{Q^{a_1}}\times \dots \times \mu_{Q^{a_t}}$ where $t\geq 2$ and all primes dividing $|M|$ are greater than $Q$. 
Let $e_1,\dots ,e_t$ generate $\mu_{Q^{a_1}}, \dots, \mu_{Q^{a_t}}$, respectively. We need to show that for only $0\%$ of $G$-extensions $\varphi$ of $k$ the torus $\Res_{K_\varphi/k}^1 \Gm$ satisfies weak approximation. By Lemma \ref{lem:WA} this happens if and only if $\im(\wedge^2 (\im \varphi_v) \to \wedge^2 G) = \{1\}$ for all places $v$ of $k$. Thus the result follows immediately from Lemma \ref{lem:WA}, Theorem \ref{thm:countWA}, and the following.

\begin{lemma}\label{lem:relaxWA}
Let $i,j\in\{1,\ldots,t\}$ be any two distinct indices with $a_i\leq a_j$. Then, for any finite place $v$ of $k$ that satisfies $\im(\wedge^2 (\im \varphi_v) \to \wedge^2 G) = \{1\} $, we have
\[e_i^{Q^{a_i-1}}\in\varphi_v(I_v) \implies \varphi_v(\Frob_v I_v)\subseteq \langle M, e_1, \dots, e_{j-1}, e_j^Q,e_{j+1} ,\dots , e_t \rangle.\]
\end{lemma}
\begin{proof}
  The element $e_i\wedge e_j$ has order $Q^{a_i}$ in $\wedge^2G$. Suppose that $e_i^{Q^{a_i-1}}\in\varphi_v(I_v)$ and let $x\in  \varphi_v(\Frob_v I_v)$. Since $\im(\wedge^2 (\im \varphi_v) \to \wedge^2 G)= \{1\} $, we have $e_i^{Q^{a_i-1}}\wedge x=1\in\wedge^2 G$. Write $x=me_1^{b_1}\dots e_t^{b_t}$ for some $m\in M$ and $b_1,\dots , b_t\in \ZZ$. Now use the fact that the $Q$-primary part of $\wedge^2G$ is
  $\bigoplus_{1\leq k<\ell\leq t}\mu_{Q^{a_k}}\otimes \mu_{Q^{a_\ell}}$,
where $\mu_{Q^{a_k}}\otimes \mu_{Q^{a_\ell}}$ is generated by $e_k\wedge e_\ell$, to see that $Q\mid b_j$.
\end{proof}

The remainder of the corrigendum is devoted to the proof of Theorem \ref{thm:countWA}. We begin with some preliminary results.

\section{Preliminaries}

\subsection{Tauberian theorem}

We require the following special case of Delange's Tauberian theorem.

\begin{theorem}\label{thm:delangeWA}
  Let $a>0$ and let
  $\mathfrak{f}(s) = \sum_{n = 1}^\infty a_n/n^s$ be a Dirichlet
	series with real non-negative coefficients which converges on $\re s > a$.
	Suppose that there exists a real number $\omega > 0$ such that 
	the function $\mathfrak{g}(s)=\mathfrak{f}(s)(s-a)^\omega$ admits an extension to a
	continuous function on the domain $\re s \geq a$ with $\mathfrak{g}(a) > 0$.
	Assume that there exists $\delta \in (0,1]$
	such that
	$$\mathfrak{g}(s) = \mathfrak{g}(a) + O((s-a)^\delta), \quad \text{as } s \to a \quad \text{for } \re s > a.$$
	If $0  <\omega < 1$ assume also that the derivative $\mathfrak{g}'$
	admits an extension to a continuous function on the domain
	$\re s \geq a$.
	Then
	$$\sum_{n \leq x} a_n \sim \frac{\mathfrak{g}(a)}{a\Gamma(\omega)}x^a(\log x)^{\omega-1}, \quad 
	\text{ as } x \to \infty.$$
      \end{theorem}

      \begin{proof}
	We prove this as an application of \cite[Thm.~I]{Del54}, closely following
	the proof of  \cite[Thm.~III, \S 5.2.1]{Del54}, which imposes the stronger assumption that
	$\mathfrak{g}$ admits a holomorphic continuation to $s=a$. 
	
	We begin with a standard reduction. The result \cite[Thm.~I]{Del54} states a Tauberian
	theorem for a suitable function  $\alpha(t)$ in terms of the analytic properties
	of its Laplace
	transform $f(s) = \int_0^\infty e^{-st} \alpha(t) \mathrm{d}t$.	
	For our application we take $\alpha(t) = \sum_{n \leq e^t} a_n$, which for $\re s > a$ gives
	$$f(s) = \int_0^\infty e^{-st} \sum_{n \leq e^t} a_n \mathrm{d}t = 
	\int_1^\infty x^{-s-1} \sum_{n \leq x} a_n \mathrm{d}x = 
	\frac{1}{s}\sum_{n =1}^\infty \frac{a_n}{n^s} = \frac{1}{s}\mathfrak{f}(s).$$
	Thus to obtain our asymptotic formula we will apply \cite[Thm.~I]{Del54}
	to $\alpha(\log x)$ (the factor $1/s$ gives rise to the factor $1/a$ in our statement).
	
	Set $g(s) = \mathfrak{g}(s)/s$. 	
	In the notation of \cite[Thm.~I]{Del54} we take $A=1$, $l$ any positive real number 
	less than $1$ and
	$$\beta(t) = 
	\begin{cases}
		0, & 0 \leq t <1, \\
		g(a)t^{\omega-1}/\Gamma(\omega), & t \geq 1.
	\end{cases}$$
 	In the notation of \emph{loc.~cit.}~we then have
	$G(s-a) = 	(s-a)^{-\omega}g(a) + h(s)$ where $h(s)$ is an entire function
	in $s$ (see \cite[Lem.~4]{Del54}). We need to consider
	$$F(s) = f(s) - G(s-a) = (s-a)^{-\omega}(g(s) - g(a)) - h(s).$$
        This admits a continuous extension to $\re s \geq a$ away from $s = a$, by our assumptions
	on $\mathfrak{g}$.

	We first consider the case where $\omega \geq 1$. In the notation of \cite[Thm.~I]{Del54}
	we take $\gamma(u) = u^{\omega-1}/g(a)$. As $s \to 0$ in the half-plane $\re s >0$  we have
	$$F(s+a) = O\left(\frac{|s|^\delta}{|s|^\omega}\right)=  O\left(\frac{|s|^{\delta-1}}{|s|^{\omega-1}}\right)$$
	thus the hypotheses of \emph{loc. cit.} hold with $\psi(r) = r^{\delta-1}$.	
	
	For the case where $0 < \omega < 1$ we take $\gamma(u) = u^{\omega}/\omega g(a)$.
	The derivative $F'(s)$ admits a continuous extension to $\re s \geq a$ away from $s = a$, 
	by our assumptions. As $s \to 0$ in the half-plane $\re s >0$ we have
	$$F'(s+a) = -\omega s^{-\omega - 1}(g(s+a) - g(a))
          + s^{-\omega}g'(s+a) - h'(s+a) =  O\left(\frac{|s|^{\delta-1}}{|s|^\omega}\right)$$
	thus the hypotheses of \emph{loc. cit.} hold with $\psi(r) = r^{\delta-1}$.	
\end{proof}

The following lemma assists with applications of the Tauberian theorem.

\begin{lemma}\label{lem:elementary_real_analysisWA}
  Suppose that $g(s)$ is holomorphic on $\re s>a$ and
  both $g$ and $g'$ extend to continuous functions on $\re s\geq a$. Define
  \begin{align*}
    g_0 : \RR&\to\CC\\
        t &\mapsto g(a+ti).
  \end{align*}
  Then $g_0$ is differentiable with derivative $g_0'(t)=ig'(a+ti)$.
\end{lemma}

\begin{proof}
  Write $g(x+iy)=u(x,y)+iv(x,y)$. For $t\in\RR$ and $h>0$ the mean value theorem shows that
  \begin{align*}
   \re(g_0(t+h)-g_0(t))&= u(a,t+h)-u(a+h, t+h)\\ &+ u(a+h, t+h)-u(a+h,t)\\ &+
                  u(a+h,t)-u(a,t)\\
    &=-hu_x(a+\xi_1,t+h) + hu_y(a+h,t+\xi_2)\\ &+ hu_x(a+\xi_3,t),
  \end{align*}
  where $\xi_1,\xi_2,\xi_3\in (0,h)$. By the continuity of $g'$ (and thus of $u_x,u_y,v_x,v_y$), we get
  that
  \begin{equation*}
    \lim_{h\to 0^+}\frac{\re(g_0(t+h)-g_0(t))}{h}=-u_x(a,t)+u_y(a,t)+u_x(a,t)=u_y(a,t),
  \end{equation*}
  as desired. Together with the analogous argument for the imaginary part, this shows that
  \begin{align*}
    \lim_{h\to 0^+}\frac{g_0(t+h)-g_0(t)}{h} &= u_y(a,t)+iv_y(a,t) = \lim_{h\to 0}\left(u_y(a+h,t)+iv_y(a+h,t)\right)\\ &= \lim_{h\to 0}ig'(a+h+ti) = ig'(a+ti),
  \end{align*}
  as desired. An analogous argument works for the left-sided limit.  
\end{proof}

\subsection{Bounds for $L$-functions}

Recall that a Dirichlet character of $k$ is a finite order Hecke
character, i.e.\ a continuous homomorphism
$\chi:\Idele/k^*\to S^1$ with
finite image. Let $\mathfrak{f}$ be the conductor of $\chi$, then
$\chi$ defines the $L$-function
$L(s;\chi)=\prod_{v\nmid\mathfrak{f}\infty}(1-\chi(\pi_v)q_v^{-s})^{-1}$
for $\re s>1$. When $\chi\neq 1$, this extends to an entire function
satisfying the usual functional equation.

\begin{lemma}\label{lem:L-boundWA}
  Let $k$ be a number field and $\chi : \Idele/k^* \to S^1$ a non-trivial Dirichlet
  character with conductor $\mathfrak{f}$. Then, for $\re s\geq 1$, we have
  \begin{equation*}
    L(s;\chi) \ll_{k,\epsilon} (N_{k/\QQ}\mathfrak{f}\cdot(1+|\im s|))^\epsilon\quad\text{ and }\quad
    \frac{L'}{L}(s;\chi) \ll_{k,\epsilon} (N_{k/\QQ}\mathfrak{f}\cdot(1+|\im s|))^{\epsilon}.
  \end{equation*}
\end{lemma}

\begin{proof}
  The first bound follows from Rademacher's classical convexity estimate, see
  e.g. \cite[III, Theorem 14A]{Moreno}.

  For the second bound, inserting the Hadamard factorisation of the completed $L$-function $\Lambda(s;\chi)$ in
  the functional equation and taking the logarithmic derivative yields the crude bound (see
  \cite[(5.28), (5.31)]{IK})
  \begin{equation*}
    \frac{L'}{L}(s;\chi) = \sum_{|s-\rho|<1}\frac{1}{s-\rho} +
    O_{k,\epsilon}((N_{k/\QQ}\mathfrak{f}\cdot(1+|\im s|))^\epsilon),
  \end{equation*}
  where the sum runs through all zeros $\rho$ of $L(s;\chi)$ with
  $|s-\rho|<1$. By \cite[(5.27)]{IK}, the number of such zeros is
  $\ll_{k,\epsilon}(N_{k/\QQ}\mathfrak{f}\cdot(1+|\im s|))^\epsilon$.
  The classical zero-free region \cite[Theorem 5.35]{IK} implies
  that for each such zero $\rho$, with at most one exception,
  \begin{equation*}
    \frac{1}{s-\rho}\ll_{k,\epsilon}(N_{k/\QQ}\mathfrak{f}\cdot(1+|\im s|))^\epsilon.
  \end{equation*}
  If there is an exceptional zero $\rho$, it is real and satisfies
  \begin{equation*}
    1-\rho \gg_{k,\epsilon}(N_{k/\QQ}\mathfrak{f})^{-\epsilon}
  \end{equation*}
  (with an ineffective implied constant) by \cite{Fogels}. These estimates
  taken together show the required bound for $\frac{L'}{L}(s;\chi)$.
\end{proof}

\section{Proof of Theorem \ref{thm:countWA}}
We now begin the proof of Theorem \ref{thm:countWA}. We may enlarge the finite
set $S$ of places of $k$, and thus assume that it contains all places dividing
$|G|$, enough places to ensure that the ring of $S$-integers $\OO_S$ of $k$ is
a PID, and all places $v$ with small $q_v$, where the meaning of ``small'' will
become clear during the proof.
All implied constants in $\ll$- and $O$-notation are allowed to depend on $k,G,S$.

\subsection{The counting function}\label{sec:f}
We will prove Theorem~\ref{thm:countWA} with $i=1$ and $j=t$. This clearly suffices.

We first translate the statement to a problem on the idele class group using class field theory. For every place $v$ of $k$ we choose a uniformiser $\pi_v \in k_v^\times$.
Write $L=  \langle M, e_1,e_2,\dots,e_{t-1},e_t^{Q}\rangle$. Consider
the function $f(\chi)=\prod_vf_v(\chi_v)$ on the group $\Hom(\Idele,G)$
of continuous homomorphisms from $\Idele$ to $G$, defined as follows: 
 \begin{equation*}
   f_v(\chi_v)=
   \begin{cases}
     1 &\text{ if } v\in S,\\
     1 &\text{ if } v\notin S\text{ and }e_1^{Q^{a_1 - 1}} \in \chi_v(\O_v^*) \implies
     \chi_v(\pi_v)\in L,\\
     0 &\text{ otherwise},
   \end{cases}
 \end{equation*} 
 where $\chi_v\in \Hom(k_v^*,G)$ is the composition of $\chi$ with the natural inclusion $k_v^*\to \Idele$.
 Note that a $G$-extension $\varphi$ of $k$ satisfying $\varphi_v\in\Lambda_v$
 for all $v\in\Omega_k\setminus S$ (as in Theorem~\ref{thm:countWA} with $i=1$
 and $j=t$) corresponds via class field theory to a continuous 
 surjective homomorphism $\chi\in\Hom(\Idele,G)$ such that $e_1^{Q^{a_1 - 1}}
 \in \chi_v(\O_v^*)$ implies $\chi_v(\pi_v\O_v^*)\subset L$ for all
 $v\in\Omega_k\setminus S$. The set of such $G$-extensions is contained in the
 set of those that correspond to continuous surjective homomorphisms $\chi\in\Hom(\Idele,G)$ with $f(\chi)=1$. These are counted by the function
 \begin{equation*}
   N^*(G,L;B) = \#\{\chi\in\Hom(\Idele/k^*,G)\ :\ \chi \text{
   surjective, } \Delta(\chi)\leq B,\ f(\chi)=1\}.
\end{equation*}

 Therefore, comparing with the total number of $G$-extensions of $k$ with discriminant bounded by $B$ (see~\cite[Theorem~1.7]{HNP}),
 it suffices to show that
\begin{equation}\label{eq:goalWA}
  N^*(G, L;B) = o(B^{1/\alpha(G)}(\log B)^{\nu(k,G)-1}),
\end{equation}
where $\alpha(G)=|G|(1-Q^{-1})$ and $\nu(k,G)=(\lvert
G[Q]\rvert-1)/[k(\mu_Q):k]$.

For any subgroup $H\subseteq G$ and $\chi:\Idele/k^*\to H$, recall the
definition of $\Phi_H(\chi)$ from \cite[\S 2]{HNP}, in particular that
$\Phi_H(\chi)=\Delta(\chi)^{|H|/|\im\chi|}$, and set
\begin{equation*}
  N(H, L;B)=\#\{\chi\in\Hom(\Idele/k^*,H)\ :\ \Phi_H(\chi)\leq B,\ f(\chi)=1\}.
\end{equation*}
Then
\begin{equation*}
  N(H,L;B)=\sum_{J\subseteq H}N^*(J,L; B^{|J|/|H|}),
\end{equation*}
and thus by M\"obius inversion as in \cite[\S 2]{HNP} and \cite[\S 2]{Wri89},
\begin{equation}\label{eq:sum_N_over_subgroupsWA}
  N^*(G,L; B) = \sum_{H\subseteq G}\mu(G/H)N(H,L; B^{|H|/|G|}).
\end{equation}

For a subgroup $H\subseteq G$, we let
$\beta_H$ be its $\FF_Q$-rank, i.e. the dimension of the $\FF_Q$-vector space
$H\otimes_\ZZ\FF_Q$. If $\beta_H<\beta_G$, then $\nu(k,H)<\nu(k,G)$ and
\begin{equation*}
  N(H,L;B^{|H|/|G|})\leq N(H; B^{|H|/|G|}) = O(B^{1/\alpha(G)}(\log B)^{\nu(k,H)-1})
\end{equation*}
by \cite[Lemma 4.7]{HNP},
where $N(H;B)=\#\{\chi\in\Hom(\Idele/k^*,H)\ :\ \Phi_H(\chi)\leq B\}$. Hence,
it suffices to consider subgroups $H$ with $\beta_H=\beta_G$ in
\eqref{eq:sum_N_over_subgroupsWA}.

\subsection{Case 1: $a_t=1$.}\label{sec:easy}
In this case $L=\langle M, e_1,e_2,\dots,e_{t-1}\rangle$. Let $H$ be a subgroup
of $G$ with $\beta_H=\beta_G$.
We show that
the individual counting function in~\eqref{eq:sum_N_over_subgroupsWA} satisfies
$N(H,L;B^{|H|/|G|})= o(B^{1/\alpha(G)}(\log B)^{\nu(k,G)-1})$.

For any finite set $T\supset S$ of places of $k$, we consider the truncation
of $f$,
\begin{equation*}
  f_T(\chi)=\prod_{v\in T}f_v(\chi_v),
\end{equation*}
 with corresponding counting
function
\begin{equation*}
  N_T(H,L;B)=\#\{\chi\in\Hom(\Idele/k^*, H)\ :\ \Phi_H(\chi)\leq
  B,\ f_{T}(\chi)=1\}.
\end{equation*}
Note that the counting function $N_T$ imposes local conditions at only finitely
many places, so the main counting results of \cite{HNP} apply. Hence,
\cite[Lemma 4.7]{HNP} shows that
\begin{equation*}
  \lim_{B\to\infty}\frac{N(H,L;B^{|H|/|G|})}{B^{1/\alpha(G)}(\log B)^{\nu(k,G)-1}} \leq  \lim_{B\to\infty}\frac{N_T(H,L;B^{|H|/|G|})}{B^{1/\alpha(G)}(\log B)^{\nu(k,G)-1}}=:c_{H,T},
\end{equation*}
with a constant $c_{H,T}\geq 0$. We will show that
$\lim_{T\to\Omega_k}c_{H,T}=0$, which is enough to ensure that
$N(H,L; B^{|H|/|G|})=o(B^{1/\alpha(G)}(\log B)^{\nu(k,G)-1})$.

Let us recall some notation from \cite{HNP}. We write $k_0:=k(\mu_Q)$. For
$\re(s)>1$, we set $\zeta_{k_0,v}(s):=1$ for archimedean $v$, and
$\zeta_{k_0,v}(s):=\prod_{w\mid v}(1-q_w^{-s})^{-1}$, the product of the local
factors of the Dedekind
zeta function of $k_0$ at all places $w$ above $v$, for non-archimedean $v$. Moreover,
$\widehat{1}_v(\cdot;s)$ is the Fourier transform of the function
$1/\Phi_H(\cdot)^s$ on $\Hom(k_v^*,H)$, and $\widehat{f}_{H,v}(\cdot;s)$ is the
Fourier transform of $f_v(\cdot)/\Phi_H(\cdot)^s$. As in \cite[(4.13)]{HNP}, we
have
\begin{equation}\label{eq:constant_bound}
  c_{H,T}\ll \lim_{s\to 1/\alpha(G)^+}\sum_{x\in k^*\otimes\dual H}\prod_{v\in T}\frac{\widehat{f}_{H,v}(x_v;s|G|/|H|)}{\zeta_{k_0,v}(\alpha(G)s)^{\nu(k,G)}}\prod_{v\notin
  T}\frac{\widehat{1}_v(x_v;s|G|/|H|)}{\zeta_{k_0,v}(\alpha(G)s)^{\nu(k,G)}}.
\end{equation}

For $x_v\in\O_v^{*} \otimes \dual{H}$, we write $x_v\in\O_v^{*Q} \otimes \dual{H}$ to say that $x_v$ is in the image of $\O_v^{*Q} \otimes \dual{H}$ in $\O_v^{*} \otimes \dual{H}$ under the natural map (which is not injective in general).

By \cite[Lemma 4.1]{HNP}, for $v\notin S$, $x_v\in\OO_v^*\otimes\dual{H}$, 
and $\re(s)\geq 0$
\begin{align*}
  & \widehat{1}_{v}(x_v;s|G|/|H|)= \\
  & \begin{cases}
    1+(Q^{\beta_G}\mathbbm{1}_{\OO_v^{*Q}\otimes
      \dual{H}}(x_v)-1)q_v^{-\alpha(G)s} + O(q_v^{-(\alpha(G)+1)s}), &
    q_v\equiv 1\bmod Q,\\
    1+ O(q_v^{-(\alpha(G)+1)s}), &q_v\not\equiv 1\bmod Q.
  \end{cases}
\end{align*}
Moreover, \cite[Lemma 3.4]{HNP} shows that
\begin{equation}\label{eq:trivial_FT_zero}
\widehat{1}_{v}(x_v;s|G|/|H|)=0\quad\text{ for }\quad x_v\notin\OO_v^*\otimes\dual{H}.
\end{equation}

For a subgroup $A$ of $k_v^*$, a subgroup $R$ of $H$, and $x_v\in k_v^*\otimes \dual{H}$, we abuse notation by writing $x_v\in A\otimes \dual{R}$ to mean that the image of $x_v$ under the natural map $k_v^*\otimes \dual{H}\to k_v^*\otimes \dual{R}$ is in the image of $A\otimes \dual{R}$ under the natural map $A\otimes \dual{R}\to k_v^*\otimes\dual{R}$.

For $\widehat{f}_{H,v}$, we have the
following result. 

\begin{lemma}\label{lem:missing_h_local_factorWA}
  Let $\re(s)\geq 0$. Write $L_H = L\cap H$ and write $V$ for the subgroup of order $Q$ of $G$ generated by $e_1^{Q^{a_1-1}}$.
  Let $v\notin S$ and 
  $x_v\in k_v^*\otimes\dual{H}$. 
  If $q_v\not\equiv 1\bmod Q$, then
  \begin{equation}\label{eq:case_1_lemma_not_one}
    \widehat{f}_{H,v}(x_v;s|G|/|H|) =  \mathbbm{1}_{\OO_v^*\otimes\dual{H}}(x_v)+ O(q_v^{-(\alpha(G)+1)s}).
  \end{equation}
  If $q_v\equiv 1\bmod Q$, then $\widehat{f}_{H,v}(x_v;s|G|/|H|)$ is equal to
  \begin{align*}
     \mathbbm{1}_{\OO_v^*\otimes\dual{H}}(x_v) + q_v^{-\alpha(G)s}\Bigl(&\mathbbm{1}_{\OO_v^*\otimes\dual{H}}(x_v)\Bigl(Q^{\beta_G}\cdot\mathbbm{1}_{\OO_v^{*Q}\otimes\dual{H}}(x_v)-Q\cdot\mathbbm{1}_{\OO_v^{*Q}\otimes\dual{ V}}(x_v)\Bigr)\\ &+ \mathbbm{1}_{\OO_v^*\otimes\dual{L_H}}(x_v)\Bigl(\mathbbm{1}_{\OO_v^{*Q}\otimes\dual{V}}(x_v)-\frac{1}{Q}\Bigr)\Bigr)+O(q_v^{-(\alpha(G)+1)s}).
  \end{align*}
Moreover, if $x_v\notin \OO_v^*\otimes\dual{L_H}$ then
$\widehat{f}_{H,v}(x_v;s) = 0$.
\end{lemma}

\begin{remark}
  Note that since $\beta_H=\beta_G$, the $Q$-torsion subgroup $G[Q]$ is contained in $H$ and hence $V\subseteq H$. 
\end{remark}

\begin{proof}[Proof of Lemma~\ref{lem:missing_h_local_factorWA}]
 Writing $k_v^*=\OO_v^*\oplus\langle\pi_v\rangle$,
  we identify $k_v^*/\OO_v^*$ with $\langle\pi_v\rangle$ and hence
  $\Hom(k_v^*,H)$ with $\Hom(\OO_v^*,H)\oplus\Hom(k_v^*/\OO_v^*,H)$. By \cite[Lemma 3.3]{HNP},
\[\widehat{f}_{H,v}(x_v;s|G|/|H|)=\sum_{l\mid (\exp(H),q_v-1)}\left(\sum_{\substack{\chi_v\in\Hom(\OO_v^*, H)\\ \ker\chi_v=\OO_v^{*l}}}\pair{\chi_v}{x_v}\tau_{f_v}(\chi_v,x_v)\right)q_v^{-|G|(1-1/l)s},\]
where 
\[ \tau_{f_v}(\chi_v, x_v) := \frac{1}{|H|}\sum_{\psi_v \in \Hom(k_v^*/\OO_v^*, H)}f_v(\chi_v\psi_v) \pair{\psi_v}{x_v}.\]
Taking $l=1$ gives $\mathbbm{1}_{\OO_v^*\otimes\dual{H}}(x_v)$. If
$q_v\not\equiv 1\bmod{Q}$, then all other terms are $O(q_v^{-(\alpha(G)+1)s})$, 
so we have shown \eqref{eq:case_1_lemma_not_one}.

Suppose that $q_v\equiv 1\bmod{Q}$. Then the contribution from $l=Q$ is
\begin{align}
&q_v^{-\alpha(G)s}\frac{1}{|H|}\sum_{\substack{\chi_v\in\Hom(\OO_v^*, H)\\ \ker\chi_v=\OO_v^{*Q}}}\pair{\chi_v}{x_v}\sum_{\psi_v \in \Hom(k_v^*/\OO_v^*, H)}f_v(\chi_v\psi_v) \pair{\psi_v}{x_v}\nonumber\\
\label{eq:e1} =&q_v^{-\alpha(G)s}\frac{1}{|H|}\sum_{\chi_v:\OO_v^*/\OO_v^{*Q}\hookrightarrow V}\pair{\chi_v}{x_v}\sum_{\psi_v \in \Hom(k_v^*/\OO_v^*, H)}f_v(\chi_v\psi_v) \pair{\psi_v}{x_v}\\
\label{eq:note1}&+q_v^{-\alpha(G)s}\frac{1}{|H|}\sum_{\substack{\chi_v:\OO_v^*/\OO_v^{*Q}\hookrightarrow H\\ \im\chi_v\not\subset V}}\pair{\chi_v}{x_v}\sum_{\psi_v \in \Hom(k_v^*/\OO_v^*, H)}f_v(\chi_v\psi_v) \pair{\psi_v}{x_v}.
\end{align}
By the definitions of $f_v$ and $L_H$, 
and our identification 
$k_v^*/\OO_v^*=\langle \pi_v\rangle$,  \eqref{eq:e1} is equal to
\begin{align}
&q_v^{-\alpha(G)s}\frac{1}{|H|}\sum_{\chi_v:\OO_v^*/\OO_v^{*Q}\hookrightarrow V}\pair{\chi_v}{x_v}\sum_{\psi_v \in \Hom(k_v^*/\OO_v^*, L_H)}\pair{\psi_v}{x_v}\nonumber\\
=&q_v^{-\alpha(G)s}\frac{|L_H|}{|H|}\mathbbm{1}_{\OO_v^*\otimes\dual{L_H}}(x_v)\sum_{\chi_v:\OO_v^*/\OO_v^{*Q}\hookrightarrow V}\pair{\chi_v}{x_v}\nonumber\\
=&q_v^{-\alpha(G)s}\frac{1}{Q}\mathbbm{1}_{\OO_v^*\otimes\dual{L_H}}(x_v)\Bigl(-1+\sum_{\chi_v:\OO_v^*/\OO_v^{*Q}\to V}\pair{\chi_v}{x_v}\Bigr)\nonumber\\
\label{eq:e12}=&q_v^{-\alpha(G)s}\frac{1}{Q}\mathbbm{1}_{\OO_v^*\otimes\dual{L_H}}(x_v)\Bigl(-1+Q\cdot\mathbbm{1}_{\OO_v^{*Q}\otimes\dual{V}}(x_v)\Bigr).
\end{align}

Furthermore, by the definition of $f_v$, \eqref{eq:note1} is equal to
\begin{align}
&q_v^{-\alpha(G)s}\frac{1}{|H|}\sum_{\substack{\chi_v:\OO_v^*/\OO_v^{*Q}\hookrightarrow H\\ \im\chi_v\not\subset V}}\pair{\chi_v}{x_v}\sum_{\psi_v \in \Hom(k_v^*/\OO_v^*, H)}\pair{\psi_v}{x_v}\nonumber\\
=&q_v^{-\alpha(G)s}\mathbbm{1}_{\OO_v^*\otimes\dual{H}}(x_v)\sum_{\substack{\chi_v:\OO_v^*/\OO_v^{*Q}\to H\\ \im\chi_v\not\subset V}}\pair{\chi_v}{x_v}\nonumber\\
=&q_v^{-\alpha(G)s}\mathbbm{1}_{\OO_v^*\otimes\dual{H}}(x_v)\Biggl(\sum_{\chi_v:\OO_v^*/\OO_v^{*Q}\to H}\pair{\chi_v}{x_v}-\sum_{\chi_v:\OO_v^*/\OO_v^{*Q}\to V}\pair{\chi_v}{x_v}\Biggr)\nonumber\\
\label{eq:note12}=&q_v^{-\alpha(G)s}\mathbbm{1}_{\OO_v^*\otimes\dual{H}}(x_v)\Biggl(Q^{\beta_G}\cdot\mathbbm{1}_{\OO_v^{*Q}\otimes\dual{H}}(x_v)-Q\cdot\mathbbm{1}_{\OO_v^{*Q}\otimes\dual{ V}}(x_v)\Biggr).
\end{align}
Taking the sum of \eqref{eq:e12} and \eqref{eq:note12} completes the proof of
the formula for $\widehat{f}_{G,v}(x_v;s)$. To see that $x_v\notin
\OO_v^*\otimes\dual{L_H}$ implies $\widehat{f}_{H,v}(x_v;s|G|/|H|) = 0$, we
also have to consider the remaining terms, i.e.~$l|(\exp(H),q_v-1)$ with $l>Q$.

Similarly to the calculations above, the contribution from $l$ is
\begin{align}
&q_v^{-\alpha(G)s}\frac{1}{|H|}\sum_{\substack{\chi_v\in\Hom(\OO_v^*, H)\\ \ker\chi_v=\OO_v^{*l}}}\pair{\chi_v}{x_v}\sum_{\psi_v \in \Hom(k_v^*/\OO_v^*, H)}f_v(\chi_v\psi_v) \pair{\psi_v}{x_v}\nonumber\\
\label{eq:e1l} =&q_v^{-\alpha(G)s}\frac{1}{|H|}\sum_{\substack{\chi_v\in\Hom(\OO_v^*, H)\\ \ker\chi_v=\OO_v^{*l}\\ V\subset \im\chi_v}}\pair{\chi_v}{x_v}\sum_{\psi_v \in \Hom(k_v^*/\OO_v^*, H)}f_v(\chi_v\psi_v) \pair{\psi_v}{x_v}\\
\label{eq:note1l}&+q_v^{-\alpha(G)s}\frac{1}{|H|}\sum_{\substack{\chi_v\in\Hom(\OO_v^*, H)\\ \ker\chi_v=\OO_v^{*l}\\ V\not\subset\im\chi_v}}\pair{\chi_v}{x_v}\sum_{\psi_v \in \Hom(k_v^*/\OO_v^*, H)}f_v(\chi_v\psi_v) \pair{\psi_v}{x_v}.
\end{align}

By the definition of $f_v$, the inner sum in~\eqref{eq:e1l} becomes
\[\sum_{\psi_v \in \Hom(k_v^*/\OO_v^*, L_H)} \pair{\psi_v}{x_v}=|L_H|\cdot \mathbbm{1}_{\OO_v^*\otimes\dual{L_H}}(x_v),\]
which is zero for $x_v\notin\OO_v^*\otimes\dual{L_H}$.
The inner sum in~\eqref{eq:note1l} becomes 
\[\sum_{\psi_v \in \Hom(k_v^*/\OO_v^*, H)} \pair{\psi_v}{x_v}=|H|\cdot\mathbbm{1}_{\OO_v^*\otimes\dual{H}}(x_v),\]
which is zero unless $x_v\in\OO_v^*\otimes\dual{H}$, which is
even stronger than $x_v\in\OO_v^*\otimes\dual{L_H}$.
\end{proof}

By \eqref{eq:trivial_FT_zero}, the summand in \eqref{eq:constant_bound} is zero
unless $x\in\OO_T^*\otimes\dual{H}$. Hence, the sum over $x$ is finite and
\begin{equation}\label{eq:constant_bound_2}
    c_{H,T}\ll \sum_{x\in \OO_T^*\otimes\dual H}\prod_{v\in
      T}\frac{\widehat{f}_{H,v}(x_v;1/\alpha(H))}{\zeta_{k_0,v}(1)^{\nu(k,G)}}\lim_{s\to
      1/\alpha(G)^+}\prod_{v\in\Omega_k\smallsetminus T}\frac{\widehat{1}_v(x_v;s|G|/|H|)}{\zeta_{k_0,v}(\alpha(G)s)^{\nu(k,G)}}.
  \end{equation}
  
  We now
  define $k_x/k_0$ to be the elementary abelian $Q$-extension associated to
  $x\in k^*\otimes\dual{H}$ as in~\cite[Section~4.4]{HNP}. Explicitly, taking a
  presentation of the $Q$-primary part of $H$ as a direct sum of cyclic groups, we can write   
\[k^*\otimes \dual{H}=k^*\otimes \dual{D}\times  k^*/k^{*Q^{c_1}}\times \dots \times  k^*/k^{*Q^{c_t}}\]
for some subgroup $D$ of $M$ and $c_1,\dots, c_t\in \ZZ_{>0}$. We write $x\in \OO_T^*\otimes \dual{H}$ as $x=(y,x_1,\dots, x_t)$ with $y\in \OO_T^*\otimes \dual{D}$ and $x_i\in \OO_T^*/\OO_T^{*Q^{c_i}}$ for $1\leq i\leq t$. Then
 \[k_x=k_0(\sqrt[Q]{x_1},\dots, \sqrt[Q]{x_t}).\]
 With this setup, it follows from \cite[Lemma 4.5]{HNP} that the function
 \begin{equation*}
 \widehat{1}_T(x;s|G|/|H|)=\prod_{v\in\Omega_k\smallsetminus
   T}\widehat{1}_v(x_v;s|G|/|H|)
\end{equation*}
 has
 a pole of order strictly smaller than
 $\nu(k,G)=(Q^{\beta_G}-1)/[k_0:k]$  at $s=1/\alpha(G)$, unless $k_x=k_0$. Hence, the summand in
 \eqref{eq:constant_bound_2} vanishes unless $k_x=k_0$. 

 \begin{lemma}
   For $x\in k^*\otimes \dual{H}$, we have $k_x=k_0$ if and only if $x\in
   k^{*Q}\otimes \dual{H}$.
 \end{lemma}

 \begin{proof}
   The ``if'' follows directly from the definition of $k_x$. For the ``only
   if'', 
   suppose $k_x=k_0$, let $i\in\{1,\ldots,t\}$ and fix a
   representative of $x_i$.
   The
   polynomial
   $X^Q-x_i$
   cannot be irreducible over $k$, as
   it has roots in $k_0$ and
   $[k_0:k]\leq Q-1$. As $Q$ is
   prime, this already implies
   $x_i\in k^{*Q}$. Moreover,
   for the first factor $k^*\otimes \dual{D}$ of $k^*\otimes \dual{H}$, we see that every $y\in k^*\otimes \dual{D}$ is a
   $Q$-th power, as $(Q,|D|)=1$.
 \end{proof}
Given $x\in k^{*Q}\otimes\dual{H}$, let $v\in \Omega_k\smallsetminus S$. By Lemma
 \ref{lem:missing_h_local_factorWA}, we have $\widehat{f}_{H,v}(x_v;s)=0$ unless
 $x_v\in\OO_v^{*}\otimes\dual{L_H}$.
 Since $\beta_H=\beta_G$, the $Q$-torsion
 subgroup $G[Q]$ is contained in $H$. Since $a_t=1$, this implies that $e_t\in
 H$ and $H=L_H\oplus \langle e_t\rangle$.
 Therefore, as $e_t$ has order $Q$ and
 $x\in k^{*Q}\otimes\dual{H}$, the condition that $x_v\in
 \OO_v^{*}\otimes\dual{L_H}$  is equivalent to 
 $x_v\in \OO_v^{*Q}\otimes\dual{H}$.
Therefore, we may further restrict the sum in \eqref{eq:constant_bound_2} to the finite
set $\OO_S^{*Q}\otimes\dual H$, which is independent of $T$. This yields
\begin{equation*}
  c_{H,T}\ll \sum_{x\in\OO_S^{*Q}\otimes\dual H}\prod_{v\in
    T\smallsetminus
    S}\frac{\widehat{f}_{H,v}(x_v;1/\alpha(H))}{\widehat{1}_v(x_v;1/\alpha(H))}\lim_{s\to
    1/\alpha(G)^+}\prod_{v\in\Omega_k\smallsetminus S}\frac{\widehat{1}_v(x_v;s|G|/|H|)}{\zeta_{k_0,v}(\alpha(G)s)^{\nu(k,G)}}.
\end{equation*}
Now observe from Lemma~\ref{lem:missing_h_local_factorWA} and the description
of $\widehat{1}_{v}(x_v;s|G|/|H|)$ preceding
Lemma~\ref{lem:missing_h_local_factorWA}, that 
\begin{align*}
 &\prod_{v\in
  T\smallsetminus
  S}\frac{\widehat{f}_{H,v}(x_v;1/\alpha(H))}{\widehat{1}_v(x_v;1/\alpha(H))} =
  \\
  &\prod_{\substack{v\in T\smallsetminus S\\q_v\not\equiv
  1\bmod Q}}\left(1+O(q_v^{-1-1/\alpha(G)})\right)
  \cdot \prod_{\substack{v\in
T\smallsetminus S\\q_v\equiv 1\bmod Q}}\left(1-(Q-2+1/Q)q_v^{-1} + O(q_v^{-1-1/\alpha(G)})\right)
\end{align*}
for $x\in\OO_S^{*Q}\otimes\dual{H}$. With, e.g., the Chebotarev density theorem, this shows that
$\lim_{T\to\Omega_k}c_{H,T}=0$, as desired. We have thus proved~\eqref{eq:goalWA}, and hence Theorem~\ref{thm:countWA}, in the case $a_t=1$.

\subsection{Case 2: $a_t\geq 2$.}
This is the hard case, where the individual counting functions
in~\eqref{eq:sum_N_over_subgroupsWA} are not
$o(B^{1/\alpha(G)}(\log B)^{\nu(k,G)-1})$. Instead, we match up the subgroups
$H$ giving the highest-order contribution in~\eqref{eq:sum_N_over_subgroupsWA}
and show cancellation between them. We already observed at the end of
Section~\ref{sec:f} that these subgroups $H$ have maximal $\FF_Q$-rank, i.e.\
$\beta_H=\beta_G$. Hence, it is enough to consider subgroups contained in the set
\begin{equation*}
W:=\{H\subseteq G\ :\ \beta_H=\beta_G,\ \mu(G/H)\neq 0\}.
\end{equation*}
Recall that $L=\langle M, e_1,e_2,\dots,e_{t-1},e_t^{Q}\rangle$ and partition $W$ into the
sets 
\begin{equation*}
 W_1:=\{H\in W\ :\ H\not\subseteq L\}\quad \text{ and }\quad W_2:=\{J\in W\ :\ J\subseteq L\}.
\end{equation*}
\begin{lemma}\label{lem:pairH}
\hfill
\begin{enumerate}
\item\label{item:quotients} For $H\in W_1$, we have $H/(H\cap L)\cong \ZZ/Q\ZZ$ and $L/(H\cap L)\cong G/H$.
\item\label{item:surj} The map $H\mapsto H\cap L$ induces a surjection $\phi: W_1\to W_2$. 
\item\label{item:bij} For $J\in W_2$, the map $\ell\mapsto \langle J, e_t\cdot \ell\rangle$ induces a bijection between $(L/J)[Q]$ and $\phi^{-1}(J)$.
\end{enumerate}
\end{lemma}

\begin{proof}
 \emph{\eqref{item:quotients}}: Since $H\in W_1$, we have $H\neq H\cap L$ and therefore the natural injection $H/(H\cap L)\hookrightarrow G/L\cong \ZZ/Q\ZZ$ is an isomorphism. Similarly, the cokernel of the natural map $L/(H\cap L)\hookrightarrow G/H$ is trivial because $G/L$ has order $Q$ and $H\not\subseteq L$.

 \emph{\eqref{item:surj} and~\eqref{item:bij}}: 
First, we claim that if $A,B\in W$ then $A\cap B\in W$. Note that having maximal $\FF_Q$-rank is equivalent to containing $G[Q]$. Moreover, suppose that $\mu(G/A)$ and $\mu(G/B)$ are nonzero. Then for every prime $p$, $G/A$ and $G/B$ contain no element of order $p^2$. Suppose for contradiction that $G/(A\cap B)$ contains an element of order $p^2$, represented by $g\in G$. Then, since $\mu(G/A)$ and $\mu(G/B)$ are nonzero, $g^p\in A$ and $g^p\in B$, so $g$ has order at most $p$ in $G/(A\cap B)$, which is the desired contradiction. Now, since $a_t\geq 2$, we have $L\in W$ and therefore the map $\phi$ is well defined.

Let $H\in \phi^{-1}(J)$, so $H\in W_1$ and $H\cap L=J$. By~\eqref{item:quotients}, the natural injection $H/J\hookrightarrow G/L\cong \ZZ/Q\ZZ$ is an isomorphism. Since $G/L$ is generated by the image of $e_t$, this implies that $H=\langle J, e_t\cdot \ell\rangle$ for some $\ell\in L$. Since $J\in W_2$, we have $\mu(G/J)\neq 0$ and hence $e_t^Q\in J$. Since $H/J$ has order $Q$, we also have $(e_t\cdot \ell)^Q\in J$, and hence $\ell^Q\in J$. So far, we have shown that any element of $\phi^{-1}(J)$ is of the form $\langle J, e_t\cdot \ell\rangle$ for some $\ell\in L$ such that $\ell^Q\in J$. Now we show that any group of this form is in $\phi^{-1}(J)$. Let $\ell\in L$ be such that $\ell^Q\in J$ and let $H=\langle J, e_t\cdot \ell\rangle$. It is clear that $H\in W_1$, since $J\in W$ and $e_t\cdot \ell\notin L$. We must show that $H\cap L=J$. Clearly, $J\subseteq H\cap L$. As noted above, since $J\in W_2$, we have $e_t^Q\in J$. So $H/J$ has order at most $Q$. Moreover, $H/(H\cap L)$ has order $Q$ by~\eqref{item:quotients}, and hence $J=H\cap L$, as required. In particular, we have shown that $\phi^{-1}(J)$ is non-empty for any $J\in W_2$, whereby $\phi$ is surjective and we have proved~\eqref{item:surj}.

To complete the proof of \eqref{item:bij}, we assert that two groups $H_1=\langle J, e_t\cdot \ell_1\rangle$ and $H_2=\langle J, e_t\cdot \ell_2\rangle$ in $\phi^{-1}(J)$ are equal if and only if $\ell_1\cdot \ell_2^{-1}\in J$. Simply observe that if $H_1=H_2$ then $\ell_1\cdot \ell_2^{-1}\in H_1\cap L=J$. The other direction is clear.
\end{proof}

Note that for $J\subseteq L$, we have $f(\chi)=1$ for all $\chi: \Idele /k^*\to J$, and hence $N(J, L; B^{|J|/|G|})=N(J; B^{|J|/|G|})$. Thus,
Lemma~\ref{lem:pairH} shows that~\eqref{eq:sum_N_over_subgroupsWA} equals
\begin{align} \label{eq:moebius_inversion_a_greater_1WAprelim}
&\sum_{H\in W_1}\mu(G/H)N(H,L; B^{|H|/|G|})+  \sum_{J\in W_2}\mu(G/J)N(J; B^{|J|/|G|})\nonumber\\
  & \quad \quad +o(B^{1/\alpha(G)}(\log B)^{\nu(k,G)-1})\nonumber\\
= &\sum_{J\in W_2}\left(\mu(G/J)N(J; B^{|J|/|G|})+\mu(L/J)\sum_{H\in\phi^{-1}(J)}N(H,L; B^{|H|/|G|})\right)\nonumber\\
 &\quad \quad+o(B^{1/\alpha(G)}(\log B)^{\nu(k,G)-1}).
  \end{align}
  
  \begin{lemma}\label{lem:mu}
  For $J\in W_2$, we have $\mu(G/J)=-|(L/J)[Q]|\cdot \mu(L/J)$.
  \end{lemma}
  
  \begin{proof}
Recall that the M\"{o}bius function on isomorphism classes of finite abelian groups satisfies $\mu(G) = 0$ if $G$ has a cyclic subgroup of order $p^n$, with $p$ a prime and $n \geq 2$, $\mu(G_1\times G_2) = \mu(G_1)\mu(G_2)$ if $G_1$ and $G_2$ have coprime order, and $\mu((\Z/p\Z)^n) = (-1)^np^{n(n-1)/2}$ for a prime $p$ and $n \in \Z_{>0}$.

Since $J\in W_2$, we have $\mu(G/J)\neq 0$ and hence $G/J\cong T\times (\Z/Q\Z)^s$ where $Q\nmid \lvert T\rvert$ and $s\in\Z_{>0}$. Since $[G:L]=Q$, it follows that $L/J$ is an index $Q$ subgroup of $G/J$, and hence $L/J\cong T\times (\Z/Q\Z)^{s-1}$. Now $\mu(G/J)=(-1)^sQ^{s(s-1)/2}\mu(T)$ and $\mu(L/J)=(-1)^{s-1}Q^{(s-1)(s-2)/2}\mu(T)$, whence the result.
  \end{proof}
  Plugging the result of Lemma~\ref{lem:mu} into~\eqref{eq:moebius_inversion_a_greater_1WAprelim} gives
  
\begin{align}   \label{eq:moebius_inversion_a_greater_1WA}
& \sum_{J\in W_2}\mu(L/J)\left(-|(L/J)[Q]|\cdot N(J; B^{|J|/|G|})+\sum_{H\in\phi^{-1}(J)}N(H,L; B^{|H|/|G|})\right) \nonumber \\
  &\quad \quad+o(B^{1/\alpha(G)}(\log B)^{\nu(k,G)-1}).
\end{align}
Lemma~\ref{lem:pairH}\eqref{item:bij} shows that the number of terms in the
sum over $H\in \phi^{-1}(J)$ is $|(L/J)[Q]|$. Our aim is thus to show that for
each of these terms we have
\begin{equation}\label{eq:cancellation_goal}
N(H,L; B^{|H|/|G|})=N(J; B^{|J|/|G|})+o(B^{1/\alpha(G)}(\log B)^{\nu(k,G)-1}),
\end{equation}
which will prove~\eqref{eq:goalWA}. We begin this proof of this, which
culminates  in  Lemma~\ref{lem:tauberian_applicationWA}, by studying
the counting functions $N(H,L;B)$ through their zeta functions
\begin{equation*}
  F_{H,f}(s)=\sum_{\chi\in\Hom(\Idele/k^*,H)}\frac{f(\chi)}{\Phi_H(\chi)^s},
\end{equation*}
converging absolutely for $\re s>1/\alpha(H)$.
Henceforth, we will fix an element $J$ of $W_2$ and an element $H$ of
$\phi^{-1}(J)$, so $J=H\cap L$. For the corresponding zeta functions, we 
prove the following proposition. Note that
$Q\cdot |J|=|H|$ by Lemma \ref{lem:pairH}\eqref{item:quotients}. 
We write again $k_0=k(\mu_Q)$.

 \begin{proposition}\label{prop:cancellationWA}
  There are holomorphic functions $g_1(s)$ and $g_2(s)$ on $\re s>1/\alpha(H)$,
  such that
  \begin{align}
    F_{H,f}(s)&=\zeta_{k_0}(\alpha(H)s)^{\nu(k,H)}\cdot g_1(s),\label{eq:f_H_f_zetaWA}\\
    F_{J,f}(Qs)&=\zeta_{k_0}(\alpha(H)s)^{\nu(k,H)}\cdot g_2(s),
  \end{align}
  for $\re s>1/\alpha(H)$. The functions $g_1,g_2$
  extend continuously to $\re s\geq 1/\alpha(H)$ and satisfy, for some
  $\delta\in(0,1]$ 
  and all $s$ in
  some compact convex neighbourhood $C$ of
  $1/\alpha(H)$ in this half-plane, the condition
  \begin{align*}
    |g_i(s)-g_i(1/\alpha(H))|\ll_C |s-1/\alpha(H)|^\delta.
  \end{align*}
  Moreover, $g_1(1/\alpha(H))=g_2(1/\alpha(H))\neq
  0$. 
\end{proposition}

Almost all that remains of this corrigendum is devoted to the proof of
Proposition \ref{prop:cancellationWA}, which we start now. The natural inclusion of $J$ in
$H$ gives a short exact sequence
\begin{equation}\label{eq:GQam1_sequenceWA}
   1\to J\xrightarrow{\text{incl}} H \xrightarrow{} H/J \to 1,
 \end{equation}
 where $H/J\cong \mu_Q$ by Lemma~\ref{lem:pairH}\eqref{item:quotients}.
 This induces an exact sequence
 \begin{equation}
   \label{eq:GQam1_hom_sequenceWA}
   1\to\Hom(\Idele/k^*,J)\xrightarrow{\text{incl}} \Hom(\Idele/k^*,H) \xrightarrow{} \Hom(\Idele/k^*,H/J).
 \end{equation}
 
  For any $\eta_v\in\Hom(k_v^*,H)$ and $\chi_v\in\Hom(k_v^*,J)$, we define
 $f_{\eta_v,v}(\chi_v):=f_v(\chi_v\eta_v)$. Similarly, for
 $\eta\in\Hom(\Idele,H)$ and $\chi\in\Hom(\Idele,J)$, we let
  \begin{equation*}
   f_{\eta}(\chi):=f(\chi\eta)=\prod_{v}f_{\eta_v,v}(\chi_v).
 \end{equation*}
 Note that,
 for all places $v\notin S$ with
 $\eta_v|_{\OO_v^*}=1$, we have
 \begin{equation*}
 f_{\eta_v,v}|_{\Hom(k_v^*/\OO_v^*,J)}=f_v|_{\Hom(k_v^*/\OO_v^*,J)}=1.
\end{equation*}

We fix a system $R\subseteq \Hom(\Idele/k^*,H)$ of representatives for the quotient
 $\Hom(\Idele/k^*,H)/\Hom(\Idele/k^*,J)$, such that $1\in R$. Then
for $\re s>1/\alpha(H)$,
 \begin{equation}
   \label{eq:sorting_sumWA}
   F_{H,f}(s)=\sum_{\eta\in R}\sum_{\chi\in\Hom(\Idele/k^*,J)}\frac{f_\eta(\chi)}{\Phi_{H}(\chi\eta)^s}.
 \end{equation}
 Our next goal is to analyse the sum over $\Hom(\Idele/k^*,J)$ via Poisson
 summation.
 
 Let $s\in\CC$. We start by computing the local Fourier transforms
 $h_{\eta_v,v}(x_v;s)$ of the functions 
 $\chi_v\mapsto f_{\eta_v,v}(\chi_v)\Phi_{H}(\chi_v\eta_v)^{-s}$ on
 $\Hom(k_v^*,J)$. With the Haar measure on $\Hom(k_v^*,J)$
 normalised as in
 \cite[\S 3.2]{HNP}, they are given for
 $x_v\in k_v^*\otimes\dual{J}$ by the
 formula
 \begin{equation}
   \label{eq:relaxed_condition_local_fourier_transformWA}
   h_{\eta_v,v}(x_v;s) = \sum_{\chi_v\in\Hom(k_v^*,J)}\frac{f_{\eta_v,v}(\chi_v)\langle\chi_v,x_v\rangle}{\Phi_{H}(\chi_v\eta_v)^s}\cdot
   \begin{cases}
     1&\text{ if $v\mid\infty$},\\
     |J|^{-1}&\text{ if $v\nmid\infty$.}
   \end{cases}
 \end{equation}
  As in 
 \cite[\S 4.3]{HNP}, we denote
 the local Fourier transform of $f_v\Phi_{J}^{-s}$ by
 $\widehat{f}_{J,v}(x_v;s)$.

  As in the proof of Lemma \ref{lem:missing_h_local_factorWA}, we
 use our uniformiser $\pi_v$ to split the sequence
 $1\to\OO_v^*\to k_v^*\to k_v^*/\OO_v^*\to 1$, thus writing
 $\Hom(k_v^*,J)=\Hom(\OO_v^*,J)\times \Hom(k_v^*/\OO_v^*,J)$. Hence, we write
 $\chi_v\in\Hom(k_v^*,J)$ as $\chi_v=\chi_{v,\text{r}}\chi_{v,\text{ur}}$ with
 $\chi_{v,\text{r}}\in\Hom(\OO_v^*,J)$ and
 $\chi_{v,\text{ur}}\in\Hom(k_v^*/\OO_v^*,J)$. 
 Recall, moreover,
 the abuse of notation introduced before the statement of Lemma \ref{lem:missing_h_local_factorWA}.

 \begin{lemma}\label{lem:relaxed_condition_halfway_local_factorsWA}
   Let $\re(s)\geq 0$, let $v$ be a place of $k$, let $\eta_v\in \Hom(k_v^*,H)$ and
   let $x_v\in
   k_v^*\otimes\dual{J}$. Write $V$ for the subgroup of order $Q$ of $G$ generated by $e_1^{Q^{a_1-1}}$.
   \begin{enumerate}
    \item If $v\notin S$ and $x_v\notin \OO_v^*\otimes \dual{J}$, then
    $h_{\eta_v,v}(x_v;s)=0$.   
  \item If $\eta_v\in\Hom(k_v^*,J)$, then
    \begin{equation*}
      h_{\eta_v,v}(x_v;s) = \langle\eta_v^{-1},x_v\rangle\widehat{f}_{J,v}(x_v;Qs).
    \end{equation*}
    
    \item If $v\notin S$ with $q_v\not\equiv 1\bmod Q$,
      $x_v\in\OO_v^*\otimes\dual{J}$ and $\eta_v|_{\OO_v^*}\in\Hom(\OO_v^*,J)$, then 
    \begin{equation*}
      h_{\eta_v,v}(x_v;s) = \langle(\eta_v|_{\OO_v^*})^{-1},x_v\rangle+O(q_v^{-(\alpha(H)+1)s}).
    \end{equation*}
    
     \item If $v\notin S$ with $q_v\equiv 1\bmod Q$, $x_v\in\OO_v^*\otimes\dual{J}$ and $\eta_v|_{\OO_v^*}\in\Hom(\OO_v^*,J)$, but $\eta_v\not\in\Hom(k_v^*,J)$, then
    \begin{align*}
      h_{\eta_v,v}(x_v;s) &= \langle(\eta_v|_{\OO_v^*})^{-1},x_v\rangle\left(1 + \left(Q^{\beta_H}\mathbbm{1}_{\OO_v^{*Q}\otimes\dual{J}}(x_v)-Q\mathbbm{1}_{\OO_v^{*Q}\otimes\dual{V}}(x_v)\right)q_v^{-\alpha(H)s}\right)\\ &+ O(q_v^{-(\alpha(H)+1)s}).
    \end{align*}
    
      \item If $v\notin S$ and
    $\eta_v|_{\OO_v^*}\not\in\Hom(\OO_v^*,J)$, then
    \begin{equation*}
      h_{\eta_v,v}(x_v;s) = O(q_v^{-(\alpha(H)+1)s}).
    \end{equation*}
    \end{enumerate}
\end{lemma}

\begin{proof}
  Note that $\beta_G=\beta_J$ implies that $V\subseteq J$.
  
  \emph{(1)}: Recall that $J=H\cap L$. For $\chi_v\in\Hom(k_v^*,H)$ and
  $\psi_v\in\Hom(k_v^*/\OO_v^*,J)$, we have
  $(\chi_v\psi_v)(\OO_v^*)=\chi_v(\OO_v^*)$ and
  $(\chi_v\psi_v)(\pi_v)\in J$ if and only if
  $\chi_v(\pi_v)\in J$. Hence,
  $f_v(\chi_v)=f_v(\chi_v\psi_v)$. Therefore, $h_{\eta_v,v}(x_v;s)$ equals
  \begin{equation*}
    \frac{1}{|J|}\sum_{\chi_{v,\text{r}}\in\Hom(\OO_v^*,J)}\frac{f_{v}(\chi_{v,\text{r}}\eta_v)\langle\chi_{v,\text{r}},x_v\rangle}{\Phi_H(\chi_{v,\text{r}}\eta_v)^s}\sum_{\chi_{v,\text{ur}}\in\Hom(k_v^*/\OO_v^*,J)}\langle\chi_{v,\text{ur}},x_v\rangle.
  \end{equation*}
  The inner sum in the last expression is $0$ for $x_v\notin \OO_v^*\otimes
  \dual{J}$ by character orthogonality.
  
    \emph{(2)}: Upon replacing $\chi_v$ by $\chi_v\eta_v$ and observing that
  $\Phi_H(\chi_v)=\Phi_{J}(\chi_v)^Q$, we see from \eqref{eq:relaxed_condition_local_fourier_transformWA} that
  \begin{align*}
    h_{\eta_v,v}(x_v;s)     &=\langle\eta_v^{-1},x_v\rangle\sum_{\chi_v\in\Hom(k_v^*,J)}\frac{f_{v}(\chi_v)\langle\chi_v,x_v\rangle}{\Phi_{J}(\chi_v)^{Qs}}\cdot
   \begin{cases}
     1&\text{ if $v\mid\infty$, }\\
     |J|^{-1}&\text{ if $v\nmid\infty$.}
   \end{cases}\\
    &=\langle\eta_v^{-1},x_v\rangle\widehat{f}_{J,v}(x_v;Qs).   
  \end{align*}
  
    \emph{(3)} and \emph{(4)}:
    Again identifying
    $k_v^*/\OO_v^*=\langle\pi_v \rangle$, we write $\eta_v=\eta_{v,\text{r}}\eta_{v,\text{ur}}$
  with $\eta_{v,r}\in\Hom(\OO_v^*,J)$ and
  $\eta_{v,\text{ur}}\in\Hom(k_v^*/\OO_v^*,H)$. Upon noting that
  $\Phi_H(\chi_v\eta_v)=\Phi_{J}(\chi_v\eta_{v,\text{r}})^Q$ and
  replacing the variable $\chi_v$ by $\chi_v\eta_{v,\text{r}}$, we see from
  \eqref{eq:relaxed_condition_local_fourier_transformWA} that
  \begin{align*}
    h_{\eta_v,v}(x_v;s)    &=\frac{\langle\eta_{v,\text{r}}^{-1},x_v\rangle}{|J|}\sum_{\chi_v\in\Hom(k_v^*,J)}\frac{f_{v}(\chi_v\eta_{v,\text{ur}})\langle\chi_v,x_v\rangle}{\Phi_{J}(\chi_v)^{Qs}}.
  \end{align*}

  Hence, $h_{\eta_v,v}(x_v;s)$ is just
  $\langle(\eta_v|_{\OO_v^*})^{-1},x_v\rangle$ times the Fourier transform of
  the function $f_{\eta_{v,\text{ur}},v}\Phi_{J}^{-Qs}$, which we may evaluate
  using \cite[Lemma 3.3]{HNP}. This shows that $h_{\eta_v,v}(x_v;s)$ is equal
  to $\langle(\eta_v|_{\OO_v^*})^{-1},x_v\rangle$ times

\begin{equation}\label{eq:relaxed_condition_sort_local_sumWA}
  \sum_{d\mid
    (\exp(J),q_v-1)}\left(\sum_{\substack{\chi_v\in\Hom(\OO_v^*, J)\\
        \ker\chi_v=\OO_v^{*d}}}\pair{\chi_v}{x_v}\tau_{f_{\eta_{v,\text{ur}},v}}(\chi_v,x_v)\right)q_v^{-|J|(1-1/d)Qs},
\end{equation}
where
\begin{equation*}
 \tau_{f_{\eta_{v,\text{ur}},v}}(\chi_v, x_v) := \frac{1}{|J|}\sum_{\psi_v \in \Hom(k_v^*/\OO_v^*, J)}f_{v}(\chi_v\psi_v\eta_{v,\text{ur}}) \pair{\psi_v}{x_v}.
\end{equation*}

Clearly, for $\psi_v\in \Hom(k_v^*/\OO_v^*, J)$ we have
$f_{v}(\psi_v\eta_{v,\text{ur}})=1$, as
$(\psi_v\eta_{v,\text{ur}})(\OO_v^*)=\{1\}$. Hence, the summand for $d=1$ in \eqref{eq:relaxed_condition_sort_local_sumWA} is equal to
\begin{equation*}
  \tau_{f_{\eta_{v,\text{ur}},v}(1,x_v)} = \frac{1}{|J|}\sum_{\psi_v \in
    \Hom(k_v^*/\OO_v^*, J)}\pair{\psi_v}{x_v}=1.
\end{equation*}
If $q_v\not\equiv 1\bmod Q$, there is no summand with $d=Q$, and the further
summands are $\ll q_v^{-(\alpha(H)+1)s}$, as desired in
\emph{(3)}.

If $q_v\equiv 1\bmod Q$, then the next summand appears for
$d=Q$. Our assumptions on $\eta_v$ in \emph{(4)} imply that $\eta_{v,\text{ur}}(\pi_v)\notin
J$. Hence, $f_v(\chi_v\psi_v\eta_{v,\text{ur}})=1$
if and only if $e_1^{Q^{a_1-1}}\notin\chi_v(\OO_v^*)$. Hence, the summand for
$d=Q$ in \eqref{eq:relaxed_condition_sort_local_sumWA} is
\begin{align*}
  &\frac{q_v^{-\alpha(H)s}}{|J|}\sum_{\substack{\chi_v\in\Hom(\OO_v^*,J)\\\ker\chi_v=\OO_v^{*Q}\\ \chi_v(\OO_v^*)\neq V}}\langle\chi_v,x_v\rangle\sum_{\psi_v\in\Hom(k_v^*/\OO_v^*,J)}\langle\psi_v,x_v\rangle\\
  =
  &q_v^{-\alpha(H)s}\left(\sum_{\chi_v\in\Hom(\OO_v^*/\OO_v^{*Q},J)}\langle\chi_v,x_v\rangle
    - \sum_{\chi_v\in\Hom(\OO_v^*/\OO_v^{*Q},V)}\langle\chi_v,x_v\rangle\right)\\
  &=q_v^{-\alpha(H)s}(Q^{\beta_H}\mathbbm{1}_{\OO_v^{*Q}\otimes
        \dual{J}}(x_v)-Q\mathbbm{1}_{\OO_v^{*Q}\otimes\dual{V}}(x_v)).
\end{align*}
Here we have repeatedly used that $\OO_v^*/\OO_v^{*Q}$ is a cyclic group of order $Q$ for $v\notin S$.

\emph{(5)}: For all $\chi_v\in \Hom(k_v^*,J)$, we have
$(\chi_v\eta_v)|_{\OO_v^*}\in \Hom(\OO_v^*,H)\smallsetminus
\Hom(\OO_v^*,J)$. 
Hence, $(\chi_v\eta_v)(\OO_v^*)$ contains an
element of order $Q^{a_t}$, which shows that
$\ker((\chi_v\eta_v)|_{\OO_v^*})=\OO_v^{*d}$ for some $d$ divisible by $Q^{a_t}$,
so in particular $Q^2\mid d$. This shows that
\begin{equation*}
  \Phi_H(\chi_v\eta_v) = \prod_{\psi\in \dual{(\OO_v^*/\OO_v^{*d})}}\Phi(\psi)^{|H|/d}=q_v^{|H|(1-1/d)}.
\end{equation*}
Hence, every summand in \eqref{eq:relaxed_condition_local_fourier_transformWA} is
$\ll q_v^{-|H|(1-1/d)s}\ll  q_v^{-(\alpha(H)+1)s}$.
\end{proof}

Note that $f_{v}(\chi_v)=1$ for all $v\notin S$ and
$\chi_v\in\Hom(k_v^*,J)$. Hence, by \cite[Lemma 4.1]{HNP}, the local Fourier transforms
$\widehat{f}_{J,v}(x_v;Qs)$
appearing in case \emph{(2)} of Lemma~\ref{lem:relaxed_condition_halfway_local_factorsWA} take, for $v\notin S$ and
$x_v\in\OO_v^*\otimes\dual{J}$, the shape
\begin{equation}\label{eq:relaxed_condition_trivial_local_fourier_transformWA}
  \begin{cases}
    1+(Q^{\beta_H}\mathbbm{1}_{\OO_v^{*Q}\otimes
      \dual{J}}(x_v)-1)q_v^{-\alpha(H)s} + O(q_v^{-(\alpha(H)+1)s}), &
    q_v\equiv 1\bmod Q,\\
    1+ O(q_v^{-(\alpha(H)+1)s}), &q_v\not\equiv 1\bmod Q.
  \end{cases}
\end{equation}

\begin{lemma}\label{lem:relaxed_condition_fourier_transformWA}
  Let $\eta\in \Hom(\Idele,H)$. 
  \begin{enumerate}
  \item Let $x=(x_v)_v\in\Idele\otimes\dual{J}$. Then the product
    $h_\eta(x;s):=\prod_{v}h_{\eta_v,v}(x_v;s)$ converges absolutely for
    $\re s>1/\alpha(H)$ and defines a holomorphic function in this half-plane.

  \item Let $\re s>1/\alpha(H)$. Then the function $x\mapsto h_\eta(x;s)$ on
    $\Idele\otimes\dual{J}$ is the Fourier transform of the function $\chi\mapsto f_\eta(\chi)\Phi_H(\chi\eta)^{-s}$ on $\Hom(\Idele,J)$.
  \end{enumerate}
\end{lemma}

\begin{proof}
  \emph{(1)}: Almost all places $v$ of $k$ satisfy $v\notin S$,
  $x_v\in\OO_v^*\otimes J$ and $\eta_v|_{\OO_v^*}=1$. For such $v$,
  Lemma \ref{lem:relaxed_condition_halfway_local_factorsWA} and \eqref{eq:relaxed_condition_trivial_local_fourier_transformWA} show that
  $h_{\eta_v,v}(x_v;s)=1+O(q_v^{-\alpha(H)s})$, which is sufficient to prove
  \emph{(1)}.

  \emph{(2)}: For all $v\notin S$ with $\eta_v|_{\OO_v^*}=1$, we see that
  $f_{\eta_v,v}(\chi_v)\Phi_H(\chi_v\eta_v)^{-s}=1$ for all
  $\chi_v\in\Hom(k_v^*/\OO_v^*,J)$. Hence, \emph{(2)} follows from
  \emph{(1)} by standard tools in Fourier analysis on restricted direct
  products, see e.g. \cite[I, Lemma 4]{Moreno}.
\end{proof}

We require the following version of Poisson summation. 

 \begin{lemma}\label{lem:relaxed_condition_poissonWA}
   Let $\re s>1/\alpha(H)$ and $\eta\in\Hom(\Idele,H)$. Then
   \begin{equation}\label{eq:poisson_formula}
     \sum_{\chi\in\Hom(\Idele/k^*,J)}\frac{f_\eta(\chi)}{\Phi_{H}(\chi\eta)^s}
     = \frac{1}{|\OO_k^*\otimes\dual{J}|}\sum_{x\in k^*\otimes\dual{J}}h_{\eta}(x;s).
   \end{equation}
   The sum on the right-hand side has only finitely many non-zero terms and
   defines a holomorphic function on $\re s>1/\alpha(H)$.
 \end{lemma}

 \begin{proof}

   We know from Lemma \ref{lem:relaxed_condition_fourier_transformWA}
   that $h_\eta(x;s)$ is the
   Fourier transform of the summand on the left-hand side.

   We consider the finite set of places
   $T=S\cup\{v\ :\ \eta_v|_{\OO_v^*}\neq 1\}$. For $v\notin T$ and $\chi_v\in\Hom(k_v^*,J)$, we have
   $\Phi_H(\chi_v\eta_v)^s=\Phi_H(\chi_v)^s=\Phi_{J}(\chi_v)^{Qs}$. As,
   moreover, the function $f_{\eta_v,v}(\cdot)$ on $\Hom(k_v^*,J)$
   is invariant under multiplication of its argument by elements of
   $\Hom(k_v^*/\OO_v^*,J)$, we see that
   \begin{align*}
     &h_{\eta_v,v}(x_v;s)=\\ &\frac{1}{|J|}\sum_{\chi_{v,\text{r}}\in\Hom(\OO_v^*,J)}\frac{f_{\eta_v}(\chi_{v,r})\langle\chi_{v,r},x_v\rangle}{\Phi_{J}(\chi_{v,r})^{Qs}}\sum_{\chi_{v,\text{ur}}\in\Hom(k_v^*/\OO_v^*,J)}\langle\chi_{v,\text{ur}},x_v\rangle.
   \end{align*}
   By character orthogonality, we obtain for
   $x_v\in\OO_v^*\otimes \dual{J}$ that
   \begin{equation*}
      h_{\eta_v,v}(x_v;s)=\sum_{\chi_{v,\text{r}}\in\Hom(\OO_v^*,J)}\frac{f_{\eta_v}(\chi_{v,r})\langle\chi_{v,r},x_v\rangle}{\Phi_{J}(\chi_{v,r})^{Qs}},
    \end{equation*}
    and otherwise $h_{\eta_v,v}(x_v;s)=0$. 
    In particular, the sum
    over $x$ on the right-hand side in \eqref{eq:poisson_formula} is finite.
    Now the desired result follows by exchanging the order of summation, as in
    the proof of \cite[Proposition 3.9]{HNP2}.
 \end{proof}

  Point \emph{(1)} of Lemma \ref{lem:relaxed_condition_halfway_local_factorsWA}
  shows that $h_\eta(x;s)=0$ for all $\eta$, unless $x\in\OO_S^*\otimes\dual{J}$.
  Continuing from  \eqref{eq:sorting_sumWA}, Lemma
  \ref{lem:relaxed_condition_poissonWA} thus shows that 
 \begin{equation}\label{eq:sorting_sum_after_poissonWA}
   F_{H,f}(s)=\frac{1}{|\OO_k^*\otimes\dual{J}|}\sum_{\eta\in R}\sum_{x\in \OO_S^*\otimes\dual{J}}h_{\eta}(x;s).
 \end{equation}

 Next, we analyse the individual Fourier transforms $h_\eta(x;s)$ in terms of
 various zeta functions.

  Recall that $k_0=k(\mu_Q)$ and let $k_x/k_0$ be the elementary abelian $Q$-extension associated to $x\in k^*\otimes\dual{J}$ as in~\cite[Section~4.4]{HNP}.   
 We also define $k_{x,1}/k_0$ to be the extension associated to $x\in k^*\otimes\dual{J}$ that is induced via Kummer theory by any homomorphism $\psi: \mu_Q\to J$ whose image is $\langle e_1^{Q^{a_1-1}}\rangle=V$. Namely, $\psi$ induces a map $\Psi\in\Hom(k^*\otimes\dual{J}, k_0^*/k_0^{*Q})$ as follows:
 \[\Psi: k^*\otimes\dual{J}\xrightarrow{\iota\otimes\dual{\psi}} k_0^*\otimes\dual{\mu_Q}=k_0^*/k_0^{*Q}, \]
 where $\iota$ denotes the natural inclusion $k^*\to k_0^*$.
  The image of $x\in k^*\otimes\dual{J}$ under $\Psi$ defines the extension $k_{x,1}/k_0$ by Kummer theory.
  
 The various conditions on $q_v$ and $x$ appearing in Lemma
 \ref{lem:relaxed_condition_halfway_local_factorsWA} can be interpreted in terms
 of splitting conditions in these extensions of $k$ as follows. Let $v\notin S$
 and $x\in \OO_S^*\otimes\dual{J}$. Then
\begin{align*}
   v \text{ splits completely in }k_0/k &\Longleftrightarrow q_v\equiv 1\bmod Q,\\
   v \text{ splits completely in }k_{x,1}/k &\Longleftrightarrow q_v\equiv 1\bmod
                                            Q\text{ and }
                                            x_v\in\OO_v^{*Q}\otimes\dual{V},\\
   v \text{ splits completely in }k_x /k &\Longleftrightarrow q_v\equiv 1\bmod
                                            Q\text{ and }
                                            x_v\in\OO_v^{*Q}\otimes\dual{J}.   
 \end{align*}
  Moreover, for $\eta\in\Hom(\Idele/k^*,H)$, let $k_\eta$ be the fixed field of $\im\eta\cap J$ inside the Galois extension of $k$ defined by $\eta$. In other words, $k_\eta$ is the extension corresponding to the image of $\eta$ under the map $\Hom(\Idele/k^*,H)\to \Hom(\Idele/k^*,H/J)$ induced by the natural map $H\to H/J$.
  Thus, $k_\eta/k$ is a Galois
 extension with $\Gal(k_\eta/k)=\im\eta/(\im\eta\cap J)\hookrightarrow H/J\cong \ZZ/Q\ZZ$, and any
 $v\notin S$ satisfies:
 \begin{align*}
   v \text{ is unramified in }k_\eta/k &\Longleftrightarrow
                                         \eta_v(\OO_v^*)\subseteq J,\\
   v \text{ splits completely in }k_\eta/k &\Longleftrightarrow
                                          \eta_v(k_v^*)\subseteq J.
 \end{align*}
 For a place $v\nmid\infty$ of $k$ and an extension $K\supset k$, we write
\begin{equation*}
  \spl(K,v):=
  \begin{cases}
    1&\text{ if $v$ splits completely in $K/k$,}\\
    0&\text{ otherwise.}
  \end{cases}
\end{equation*}

\begin{lemma}\label{lem:loca_factors_with_splitting_conditionsWA}
  Let $\re s\geq 0$, $v\notin S$, $x\in\OO_S^*\otimes\dual{J}$ and
  $\eta\in\Hom(\Idele/k^*,H)$ such that
  $\eta_v|_{\OO_v^*}\in\Hom(\OO_v^*,J)$. Then
  \begin{align*}
    h_{\eta_v,v}(x_v;s) = &\langle(\eta_v|_{\OO_v^*})^{-1},x_v\rangle\Big(1\\
    + &\left(Q^{\beta_H}\spl(k_x,v)-Q\spl(k_{x,1},v)+Q\spl(k_\eta
    k_{x,1},v)-\spl(k_\eta k_0,v)\right)q_v^{-\alpha(H)s}\\ + &O(q_v^{-(\alpha(H)+1)s})\Big).
  \end{align*}
\end{lemma}

\begin{proof}
  This follows from Lemma \ref{lem:relaxed_condition_halfway_local_factorsWA}
  together with the description of $\widehat{f}_{J,v}(x_v;Qs)$ in
  \eqref{eq:relaxed_condition_trivial_local_fourier_transformWA} and the
  characterisation of the various splitting conditions immediately preceding the statement of Lemma~\ref{lem:loca_factors_with_splitting_conditionsWA}.
\end{proof}

 For any Galois extension $K/k$, let $\zeta_K(s)$ be its Dedekind zeta function. For an
archimedean place $v\in\Omega_k$, let $\zeta_{K,v}(s):=1$. For $v$
non-archimedean, we define the local factor at $v$ of $\zeta_K$ as
\begin{equation*}
  \zeta_{K,v}(s):=\prod_{\substack{w\in\Omega_K\\w\mid
      v}}\frac{1}{1-q_v^{-f(v)s}}=
    1+\#\{w\mid v\}q_v^{-f(v)s} + O_{[K:k]}\left(q_v^{-2f(v)s}\right),
  \end{equation*}
where $f(v)$ is the inertia degree of $v$ in $K$. Lemma~\ref{lem:loca_factors_with_splitting_conditionsWA} shows
that, for $\re s> 1$,
$v\notin S$, $x\in\OO_S^*\otimes\dual{J}$ and
$\eta\in\Hom(\Idele/k^*,H)$ with $\eta_v|_{\OO_v^*}\in\Hom(\OO_v^*,J)$, we
have
\begin{equation}\label{eq:hLfunction}
  h_{\eta_v,v}(x_v;s)=\langle(\eta_v|_{\OO_v^*})^{-1},x_v\rangle \left(L_v(\eta,x;\alpha(H)s) + O(q_v^{-(\alpha(H)+1)s})\right),
\end{equation}
where $L_v(\eta,x;s)$ is the local factor at $v$ of the function
\begin{equation*}
  L(\eta,x;s):=\zeta_{k_x}(s)^{\frac{Q^{\beta_H}}{[k_x:k]}}\zeta_{k_{x,1}}(s)^{\frac{-Q}{[k_{x,1}:k]}}\zeta_{k_\eta
    k_{x,1}}(s)^{\frac{Q}{[k_\eta k_{x,1}:k]}}\zeta_{k_\eta
    k_{0}}(s)^{\frac{-1}{[k_\eta k_0:k]}}.
\end{equation*}
Write
\begin{equation*}
  \nu(\eta,x):=\frac{1}{[k_{x,1}:k]}\left(\frac{Q^{\beta_H}}{[k_x:k_{x,1}]}-Q\right)
  + \frac{1}{[k_\eta k_0:k]}\left(\frac{Q}{[k_\eta k_{x,1}:k_\eta k_0]}-1\right).
\end{equation*}
As $[k_x:k_{x,1}]\mid Q^{\beta_H-1}$ and $[k_\eta k_{x,1}:k_\eta k_0]\mid Q$,
it is clear that $\nu(\eta,x)\geq 0$.
 
 \begin{lemma}\label{lem:nu_analysisWA}
  Let $\eta\in\Hom(\Idele/k^*,H)$ and $x\in\OO_S^*\otimes\dual{J}$. Then we have
  \begin{equation*}
    \nu(\eta,x)\leq \frac{Q^{\beta_H}-1}{[k_0:k]} = \nu(k,H),
  \end{equation*}
  with equality precisely when $\eta\in\Hom(\Idele/k^*,J)$ and $k_x=k_0$.
\end{lemma}

\begin{proof}
  Directly from the definition we see that $\nu(\eta,x)$ is maximal if and only
  if all the
  occurring degrees are minimal. 
  Hence the upper bound,
  which is attained precisely when $k_x=k_0$ and $k_\eta\subseteq k_0$. As
  $\Gal(k_\eta/k)=\im \eta/(\im \eta\cap J)\hookrightarrow H/J\cong\ZZ/Q\ZZ$ and $[k_0:k]\mid Q-1$, the latter
  condition holds if and only if
  $k_\eta=k$, which means that $\im\eta\subseteq J$. 
\end{proof}

For $\eta\in\Hom(\Idele/k^*,H)$, we write
\begin{equation*}
  T(\eta):=\{v\in\places\ :\ v\notin S,\ \eta_v|_{\OO_v^*}\notin\Hom(\OO_v^*,J)\}.
\end{equation*}
Note that $T(\eta)$ is the set of places $v\notin S$ where $k_\eta/k$ is ramified, in particular it is finite.

\begin{lemma}\label{lem:prelim_zeta_function_analysisWA}
  Let $\eta\in\Hom(\Idele/k^*,H)$ and $x\in\OO_S^*\otimes\dual{J}$. Then
  there is a function $g_0(\eta,x;s)$, holomorphic in an open neighbourhood of the
  half-plane $\re s\geq 1$, such that, for $\re s>1$,
  \begin{equation*}
    L(\eta,x;s) = \zeta_{k_0}(s)^{\nu(\eta,x)}\cdot g_0(\eta,x;s).
  \end{equation*}
  Moreover, for any $\epsilon>0$ the function $g_0(\eta,x;s)$ satisfies on
  $\re s\geq 1$ the bounds
  \begin{align*}
    g_0(\eta,x;s)&\ll_\epsilon(1+|\im s|)^\epsilon\prod_{v\in
                 T(\eta)}q_v^{\epsilon},\\
    g_0'(\eta,x;s)&\ll_\epsilon(1+|\im s|)^\epsilon\prod_{v\in
                 T(\eta)}q_v^{\epsilon}.
  \end{align*}
\end{lemma}

\begin{proof}
  We start by observing that, for 
  $\re s > 1$,
  \begin{equation}\label{eq:L_decompositionWA}
  \begin{aligned}
    \frac{L(\eta,x;s)}{\zeta_{k_0}(s)^{\nu(\eta,x)}}&=
                 \left(\frac{\zeta_{k_x}}{\zeta_{k_0}}(s)\right)^{\frac{1}{[k_{x,1}:k]}\left(\frac{Q^{\beta_H}}{[k_x:k_{x,1}]}-Q\right)}
                 \left(\frac{\zeta_{k_x}}{\zeta_{k_{x,1}}}(s)\right)^{\frac{Q}{[k_{x,1}:k]}}\\ &
    \left(\frac{\zeta_{k_\eta k_{x,1}}}{\zeta_{k_0}}(s)\right)^{\frac{1}{[k_\eta
        k_0:k]}\left(\frac{Q}{[k_\eta k_{x,1}:k_\eta k_0]}-1\right)}
    \left(\frac{\zeta_{k_\eta k_{x,1}}}{\zeta_{k_\eta
          k_0}}(s)\right)^{\frac{1}{[k_\eta k_0:k]}}.
  \end{aligned}
\end{equation}
  Recall that for a tower $F/K/k_0$ of number fields with
  $F/k_0$ abelian, we have
  \begin{equation*}
    \frac{\zeta_F}{\zeta_K}(s) =  \prod_{\chi}L(s;\chi),
  \end{equation*}
  where $\chi$ runs through the characters of
  $\Adele_{k_0}^*/k_0^* N_{F/k_0}\Adele_F^*$ that are non-trivial on
  $N_{K/k_0}\Adele_K^*$. In particular, each of the occurring
  $L$-functions is entire. Next, we observe that the Kummer extensions
  $k_x/k_0$ and $k_\eta k_{x,1}/k_0$ are ramified only at places of $k_0$ lying above places in
  $S\cup T(\eta)$. Outside $S$, the ramification is tame, hence the conductor
  of each of these extensions has absolute norm
  $\ll \prod_{v\in T(\eta)}q_v^{[k_0 : k]}\ll \prod_{v\in
    T(\eta)}q_v^{Q-1}$.

  Each appearing $L$-function $L(s;\chi)$ is zero-free on some open
  neighbourhood of the half-plane $\re s\geq 1$, and thus has holomorphic
  $[k_{x,1} k_\eta : k]$-th roots. Hence, the function
  \begin{equation*}
    g_0(\eta,x;s) = \frac{L(\eta,x;s)}{\zeta_{k_0}(s)^{\nu(\eta,x)}},
  \end{equation*}
  as a product of non-negative integer powers of such roots for all occurring $L(s;\chi)$, is
  holomorphic on an open neighbourhood of $\re s\geq 1$.

  Now let $\re s\geq 1$. Lemma \ref{lem:L-boundWA}, applied to all the $L(s;\chi)$, shows that
  \begin{equation*}
    g_0(\eta,x;s)\ll_\epsilon (1+|\im s|)^\epsilon\prod_{v\in T(\eta)}q_v^{\epsilon}.
  \end{equation*}
  Using Lemma \ref{lem:L-boundWA} again, the simple observation that, for
  arbitrary exponents $\beta\in\RR$,
   \begin{equation}\label{eq:derivative_boundWA}
    (L(s;\chi)^{\beta})' =
    \beta\frac{L'}{L}L^{\beta}(s;\chi),
  \end{equation}
  and the product rule, we see that also
  \begin{equation*}
    g_0'(\eta,x;s)\ll_\epsilon (1+|\im s|)^\epsilon\prod_{v\in T(\eta)}q_v^{\epsilon},
  \end{equation*}
  as desired.
\end{proof}

\begin{lemma}\label{lem:zeta_function_analysisWA}
  Let $\eta\in\Hom(\Idele/k^*,H)$ and $x\in\OO_S^*\otimes\dual{J}$. Then
  there is a function $g(\eta,x;s)$, holomorphic in
  an open neighbourhood of
  $\re s\geq 1/\alpha(H)$, such that, for $\re s>1/\alpha(H)$,
  \begin{equation*}
    h_\eta(x;s) = \zeta_{k_0}(\alpha(H)s)^{\nu(\eta,x)}\cdot g(\eta,x;s).
  \end{equation*}
  Moreover, for any $\epsilon>0$, the function $g(\eta,x;s)$ 
  satisfies on $\re s\geq 1/\alpha(H)$ the bounds 
  \begin{align}
    g(\eta,x;s)&\ll_\epsilon(1+|\im s|)^\epsilon\prod_{v\in
                 T(\eta)}q_v^{-(\alpha(H)+1)s+\epsilon}\label{eq:g_prod_boundWA}\\
    g'(\eta,x;s)&\ll_\epsilon(1+|\im s|)^\epsilon\prod_{v\in
                 T(\eta)}q_v^{-(\alpha(H)+1)s+\epsilon}.\label{eq:gprime_prod_boundWA}
  \end{align}
\end{lemma}

\begin{proof}
  From Lemma \ref{lem:relaxed_condition_halfway_local_factorsWA} and
  \eqref{eq:hLfunction},
  we see that $h_\eta(x;s)$
  takes the form
  \begin{align*}
     &\prod_{v\in S}h_{\eta_v,v}(x_v;s)\prod_{v\in
    T(\eta)}O(q_v^{-(\alpha(H)+1)s})\prod_{\substack{v\notin S\cup
    T(\eta)\\\eta_v|_{\OO_v^*}\neq
    1}}\langle(\eta_v|_{\OO_v^*})^{-1},x_v\rangle\\ &\prod_{v\notin S\cup T(\eta)}(L_v(\eta,x;\alpha(H)s)+O(q_v^{-(\alpha(H)+1)s})).
  \end{align*}
  If $\re s>1/(\alpha(H)+1)$,  then $h_{\eta_v,v}(x_v;s)\ll 1$ and
  $1\ll \zeta_{K/k_0,v}(\alpha(H)s)\ll 1$ for all $v\in\Omega_k$ and all Galois
  extensions $K/k_0$ of degree $[K:k_0]\leq Q^{\beta_H+1}$, which covers all
  extensions occurring in the definition of $L(\eta,x;s)$. Therefore, there is a
  function $g_1(\eta,x;s)$, holomorphic on $\re(s)>1/(\alpha(H)+1)$ and
  satisfying
  \begin{equation}\label{eq:g1_prod_boundWA}
    g_1(\eta,x;s) \ll_\epsilon \prod_{v\in T(\eta)}q_v^{-(\alpha(H)+1)s+\epsilon}
  \end{equation}
  on $\re(s)\geq 1/(\alpha(H)+1)+\epsilon$ for any small $\epsilon>0$, such
  that, for $\re(s)>1/\alpha(H)$,
  \begin{equation*}
    h_\eta(x;s) = L(\eta,x;\alpha(H)s)\cdot g_1(\eta,x;s) = \zeta_{k_0}(\alpha(H)s)^{\nu(\eta,x)}g_0(\eta,x;\alpha(H)s)g_1(\eta,x;s),
  \end{equation*}
  where $g_0(\eta,x;s)$ is as in Lemma
  \ref{lem:prelim_zeta_function_analysisWA}.
  Hence, the function
  \begin{equation*}
    g(\eta,x;s):=g_0(\eta,x;\alpha(H)s)g_1(\eta,x;s)
  \end{equation*}
  is holomorphic in an open neighbourhood of $\re s\geq 1/\alpha(H)$ and satisfies
  \eqref{eq:g_prod_boundWA}.
  Next, we use Cauchy's integral formula to bound
  $g_1'(\eta,x;s)$ in terms of our bound \eqref{eq:g1_prod_boundWA} for $g_1(\eta,x;s)$. Together with the
  bounds for $g_0$ and $g_0'$ from Lemma
  \ref{lem:prelim_zeta_function_analysisWA} and the product rule, this shows
  \eqref{eq:gprime_prod_boundWA}.
\end{proof}

\begin{proof}[Completion of the proof of Proposition \ref{prop:cancellationWA}]
  We start with the expression \eqref{eq:sorting_sum_after_poissonWA} for
  $F_{H,f}(s)$. Recall that our representative in $R$ for the trivial class
  $\Hom(\Idele/k^*,J)$ is $\eta=1$.

  Consider the finite set
  \begin{equation*}
    M:=\{\nu(\eta,x)\ :\ \eta\in\Hom(\Idele/k^*,H),\ x\in\OO_S^*\otimes\dual{J}\}.
  \end{equation*}
  We sort the values of $\eta$ and $x$ in \eqref{eq:sorting_sum_after_poissonWA}
  according to the value of $\nu(\eta,x)$ and obtain from Lemma
  \ref{lem:zeta_function_analysisWA}, for $\re s>1/\alpha(H)$, that
  \begin{equation}\label{eq:f_H_f_different_zetaWA}
    F_{H,f}(s) = \sum_{\nu\in M}\zeta_{k_0}(\alpha(H)s)^{\nu}g_\nu(s),
  \end{equation}
  where
  \begin{equation*}
    g_\nu(s) := \frac{1}{|\OO_k^*\otimes\dual{J}|}\sum_{\substack{x\in\OO_S^*\otimes\dual{J}}}\sum_{\substack{\eta\in R\\\nu(\eta,x)=\nu}}g(\eta,x;s).
  \end{equation*}
  Let us fix $\nu$ and study the function $g_\nu(s)$.  The sum over $x$ is
  finite, 
  so let us focus on the sum over $\eta$.

  Note that the map induced by $H\to H/J$ embeds $R$ into
  $\Hom(\Idele/k^*,H/J)$. Denote the image of $\eta$ under this map by $\bar{\eta}$ and recall that $k_\eta/k$ is the extension corresponding to $\bar{\eta}$.
  The places in $T(\eta)$ are exactly those places
  not in $S$ where $k_\eta/k$ is ramified. 
  Therefore, $\prod_{v\in T(\eta)}q_v\asymp\Phi(\bar{\eta})$, the
  conductor of $\bar{\eta}$. We obtain, for
  $\re(s)\geq 1/\alpha(H)$, that
  \begin{equation*}
    \prod_{v\in T(\eta)}q_v^{-(\alpha(H)+1)\re s+\epsilon}\ll \Phi(\bar{\eta})^{-(1+1/\alpha(H))+\epsilon}.
  \end{equation*}
  Recall that $H/J\simeq \mu_Q$ by Lemma~\ref{lem:pairH}\eqref{item:quotients}. For sufficiently small $\epsilon$, the sum
  \begin{equation*}
    \sum_{\psi\in\Hom(\Idele/k^*,\mu_Q)}\Phi(\psi)^{-(1+1/\alpha(H))+\epsilon}
  \end{equation*}
  converges, as it is equal to the value at $s=1+1/\alpha(H)-\epsilon$ of the Dirichlet
  series
  \begin{equation*}
    \sum_{\psi\in\Hom(\Idele/k^*,\mu_Q)}\frac{1}{\Phi(\psi)^s},
  \end{equation*}
  which converges absolutely for $\re s> 1$ by \cite[Lemma
  2.10]{Woo10}. By the dominated convergence theorem, the function $g_\nu(s)$
  and its derivative extend continuously
  to $\re s\geq 1/\alpha(H)$. Now we take
  \begin{equation}\label{eq:g1_def}
    g_1(s) := \sum_{\nu\in M}\zeta_{k_0}(\alpha(H)s)^{\nu-\nu(k,H)}g_\nu(s),
  \end{equation}
  so \eqref{eq:f_H_f_zetaWA} follows from \eqref{eq:f_H_f_different_zetaWA}. As
  $\zeta_{k_0}^{-1}(\alpha(H)s)$ has no poles and only the zero at
  $s=1/\alpha(H)$ in a neighbourhood of $\re s\geq 1/\alpha(H)$, the negative
  powers of $\zeta_{k_0}(\alpha(H)s)$ in \eqref{eq:g1_def} extend continuously
  to $\re s\geq 1/\alpha(H)$. Hence, the same holds for $g_1(s)$.

  Let $\delta_0:=\nu(k,H)-\max(M\smallsetminus\{\nu(k,H)\})>0$ be the difference
  between $\nu(k,H)$ and the second biggest element of $M$. We set
  $a:=1/\alpha(H)$ and show that in a small
  compact convex neighbourhood $C$ of $a$ in $\re s\geq a$, each summand of $g_1(s)$ satisfies
  \begin{equation}\label{eq:summand_hoelder_boundWA}
    |\zeta_{k_0}(\alpha(H)s)^{\nu-\nu(k,H)}g_\nu(s)-\zeta_{k_0}(1)^{\nu-\nu(k,H)}g_\nu(a)|\ll
    _C |s-a|^{\min\{1,\delta_0\}}.
  \end{equation}
  Then clearly the same will follow for $g_1(s)$, as desired. Let $s\in C$. In case
  $\nu=\nu(k,H)$, we consider the function $g_0(t):=g_{\nu(k,H)}(a+t(s-a))$
  on $[0,1]$. This function is differentiable on $(0,1)$ with derivative
  \begin{equation*}
    g_0'(t)=(s-a)g_{\nu(k,H)}'(a+t(s-a))\ll_C |s-a|.
  \end{equation*}
  Note that to obtain this when $\re s=a$, we have used
  Lemma \ref{lem:elementary_real_analysisWA}. By the mean value theorem,
  \begin{equation*}
    |g_{\nu(k,H)}(s)-g_{\nu(k,H)}(a)| = |g_0(1)-g_0(0)|\leq g_0'(\xi)\ll_C |s-a|,
  \end{equation*}
  where $\xi\in (0,1)$. This shows \eqref{eq:summand_hoelder_boundWA} in the case
  $\nu=\nu(k,H)$. When $\nu<\nu(k,H)$, then $\zeta_{k_0}(1)^{\nu-\nu(k,H)}=0$
  and
  \begin{equation*}
    |\zeta_{k_0}(\alpha(H)s)^{\nu-\nu(k,H)}g_\nu(s)|\ll_C |\zeta_{k_0}(\alpha(H)s)^{-1}|^{\min\{1,\delta_0\}},
  \end{equation*}
  if $C$ is sufficiently small so that $|\zeta_{k_0}(\alpha(H)s)^{-1}|\leq 1$
  in $C$. Hence, \eqref{eq:summand_hoelder_boundWA} follows from the bound
  \begin{equation*}
    |\zeta_{k_0}(\alpha(H)s)^{-1}| \ll_C |s-a|,
  \end{equation*}
  coming again from the mean value theorem and the fact that
  $\zeta_{k_0}(\alpha(H)s)^{-1}$ is holomorphic in a neighbourhood of $\re s\geq a$.

 Next, we isolate the summands for the representative $\eta=1$ of the trivial class,
  \begin{equation*}
    g_2(s) := \frac{1}{|\OO_k^*\otimes\dual{J}|}\sum_{x\in\OO_S^*\otimes\dual{J}}\zeta_{k_0}(\alpha(H)s)^{\nu(1,x)-\nu(k,H)}g(1,x;s).
  \end{equation*}
  This sum is finite, and Lemma \ref{lem:relaxed_condition_poissonWA} shows
  that
  for $\re s>1/\alpha(H)$,
  \begin{align*}
    \zeta_{k_0}(\alpha(H)s)^{\nu(k,H)}\cdot g_2(s) &=
    \frac{1}{|\OO_k^*\otimes\dual{J}|}\sum_{x\in
                                                     \OO_S^*\otimes\dual{J}}h_{1}(x;s)\\
                                                   &=\sum_{\chi\in\Hom(\Idele/k^*,J)}\frac{f(\chi)}{\Phi_{H}(\chi)^s}
                                                     = F_{J,f}(Qs),
  \end{align*}
  as desired.
   Finally, we observe that
  \begin{align*}
    g_1(1/\alpha(H)) &= g_2(1/\alpha(H))\\ &+
    \frac{1}{|\OO_k^*\otimes\dual{J}|}\sum_{\substack{\eta\in R\\\eta\neq
        1}}\sum_{x\in\OO_S^*\otimes\dual{J}}\zeta_{k_0}(1)^{\nu(\eta,x)-\nu(k,H)}g(\eta,x;1/\alpha(H))\\ &=
    g_2(1/\alpha(H)),
  \end{align*}
  by Lemma \ref{lem:nu_analysisWA}. We still need to argue that
  $g_2(1/\alpha(H))\neq 0$. For this, note that $f=1$ on all of
  $\Hom(\Idele/k^*,J)$, as $J\subset L$, and moreover that
  $\beta_J=\beta_G=\beta_H$. From \cite[Proposition 5.5]{Wri89} or \cite[Lemma
  4.7]{HNP}, we know that
  \begin{equation*}
     F_{J,f}(Qs) = \sum_{\chi\in\Hom(\Adele^*/k^*,J)}\frac{1}{\Phi_J(\chi)^{Qs}}
   \end{equation*}
   has a pole of order at most $\nu(k,J)=\nu(k,H)$ at $s=1/(Q\alpha(J)) =
   1/\alpha(H)$. On the other hand,
   \begin{equation*}
     \{\chi\in\Hom(\Adele^*/k^*,J)\ :\ \Phi_J(\chi)\leq B^{1/Q}\} \supseteq
     \{\chi\in J\text{-ext}(k)\ :\ \ \Delta(\chi)\leq B^{1/Q}\},
   \end{equation*}
   and the counting function of the latter set grows with order
   $B^{1/\alpha(H)}(\log B)^{\nu(k,J)-1}$ by \cite[Theorem I.2]{Wri89}. Hence, the pole can not be of order
   smaller than $\nu(k,J)$.
 \end{proof}
 
 What remains is to deduce the desired cancellation
 \eqref{eq:cancellation_goal} from Proposition~\ref{prop:cancellationWA}.

 \begin{lemma}\label{lem:tauberian_applicationWA}
  We have
  \begin{equation*}
    N(H,L;B^{|H|/|G|}) = N(J;B^{|J|/|G|)}) +
    o(B^{1/\alpha(G)}(\log B)^{\nu(k,G)-1}).
  \end{equation*}
\end{lemma}

\begin{proof}
We apply Delange's Tauberian theorem
  (Theorem \ref{thm:delangeWA}) to the results of Proposition
  \ref{prop:cancellationWA}.

  We take $a:=1/\alpha(H)$ and $\omega:=\nu(k,H)\in\ZZ_{\geq 1}$. Moreover, we write
  $u(s):=((s-a)\zeta_{k_0}(\alpha(H)s))^\omega$. As $\zeta_{k_0}(\alpha(H)s)$
  has a pole of order $1$ at $s=a$ and no other poles or zeros in an open
  neighbourhood of $\re(s)\geq a$, $u(s)$ extends to a holomorphic function
  on this neighbourhood with $u(a)\neq 0$. For $\mathfrak{f}(s)=F_{H,f}(s)$, we obtain
  \begin{equation*}
    \mathfrak{g}(s)=(s-a)^\omega \mathfrak{f}(s) = u(s)g_1(s), 
  \end{equation*}
  which is holomorphic on $\re s>a$ and extends continuously to
  $\re s\geq a$. We have
  \begin{equation*}
    \mathfrak{g}(a)=u(a) g_1(a)\neq 0.
  \end{equation*}
  Moreover, fix a sufficiently small compact convex neighbourhood $C$ of $a$ in
  $\re s\geq a$. For any $s\in C$, we get
  \begin{align*}
    |\mathfrak{g}(s)-\mathfrak{g}(a)| &\leq |u(s)|\cdot|g_1(s)-g_1(a)| +|g_1(a)|\cdot|u(s)-u(a)|\\ &\ll_C |s-a|^{\delta}. 
  \end{align*}
  We have verified the requirements of Theorem
  \ref{thm:delangeWA}, and hence
  \begin{equation*}
    N(H,L;B^{|H|/|G|})=\sum_{\Phi_H(\chi)\leq B^{|H|/|G|}}f(\chi)\sim
    \frac{\mathfrak{g}(a)}{a\Gamma(\omega)}\left(\frac{|H|}{|G|}\right)^{\omega-1}B^{1/\alpha(G)}(\log B)^{\omega-1}.
  \end{equation*}
  Noting that $f(\chi)=1$ for all $\chi: \Idele /k^*\to J$, an analogous argument with $\mathfrak{f}(s)=F_{J,f}(Qs)$ shows that
   \begin{equation*}
    N(J;B^{|J|/|G|})=\sum_{\Phi_{J}(\chi)^Q\leq B^{|H|/|G|}}f(\chi)\sim
    \frac{\mathfrak{g}(a)}{a\Gamma(\omega)}\left(\frac{|H|}{|G|}\right)^{\omega-1}B^{1/\alpha(G)}(\log B)^{\omega-1},
  \end{equation*}
  indeed with the same value of $\mathfrak{g}(a)$, as $g_1(1/\alpha(H))=g_2(1/\alpha(H))$
  by Proposition~\ref{prop:cancellationWA}.
\end{proof}

Plugging the results of Lemma~\ref{lem:tauberian_applicationWA} into~\eqref{eq:moebius_inversion_a_greater_1WA} shows~\eqref{eq:goalWA} in the case $a_t\geq 2$, thereby completing the proof of Theorem~\ref{thm:countWA}.

\begin{ack}
Frei was supported by EPSRC Grant EP/T01170X/1 and EP/T01170X/2. Loughran was supported by UKRI Future Leaders Fellowship MR/V021362/1. Newton was supported by EPSRC Grant EP/S004696/1 and EP/S004696/2, and UKRI Future Leaders Fellowship MR/T041609/1 and MR/T041609/2
\end{ack}


\begin{thebibliography}{xx}
\bibitem{Del54}
{H. Delange}, 
{G\'{e}n\'{e}ralisation du Th\'{e}or\`{e}me de Ikehara}.
\emph{Ann. Sc. E.N.S.} (3) {\bf71} (1954), 213--242.

\bibitem{Fogels}
{E. Fogels},
{\"Uber die Ausnahmenullstelle der Heckeschen \(L\)-Funktionen},
\emph{Acta Arith.} \textbf{8} (1963), 307--309.


\bibitem{HNP}
{C. Frei, D. Loughran, R. Newton}, 
{The Hasse norm principle for abelian extensions}. 
\emph{Amer. J. Math.} \textbf{140}(6)
(2018), 1639--1685.

\bibitem{HNP2}
{C. Frei, D. Loughran, R. Newton}, 
{Number fields with prescribed norms}. 
\emph{Comment. Math. Helv.} \textbf{97}(1)
(2022), 133--181

\bibitem{IK}
{H. Iwaniec, E. Kowalski},
\emph{Analytic number theory}.
American Mathematical Society Colloquium Publications, 53,
American Mathematical Society, Providence, RI, 2004.

\bibitem{LMFDB}
The LMFDB Collaboration, \emph{The L-functions and Modular Forms Database}.
http://www.lmfdb.org, 2020, [Online; accessed 3 April 2020].

\bibitem{KR1}
P. Koymans, N. Rome,
Weak approximation on the norm one torus. \texttt{arXiv:2211.05911}.

\bibitem{KR2}
P. Koymans, N. Rome,
A note on the Hasse norm principle. \texttt{arXiv:2301.10136}.

\bibitem{Moreno}
{C.J.~Moreno}, 
\emph{Advanced analytic number theory: {$L$}-functions}.
Mathematical Surveys and Monographs, 115,
American Mathematical Society, Providence, RI, 2005.

\bibitem{Woo10} 
{M.~M.~Wood}, 
{On the probabilities of local behaviors in abelian field extensions}.
{\em Compositio Math.} {\bf 146} (2010), no.~1, 102--128.

\bibitem{Wri89} 
{D.~Wright}, 
{Distribution of discriminants of abelian extensions}.
{\em Proc. London Math. Soc.} {\bf 58} (1989), no.~1, 17--50.
\end{thebibliography}
\end{document}